\documentclass[a4paper,twosided,11pt]{article}

\usepackage[
bindingoffset=0.3in,
bottom=1.5in,
top=1.2in,
left=1in,
right=1in]{geometry}

\usepackage[OT1]{fontenc}
\usepackage{amsthm,amsmath}
\usepackage{natbib}
\usepackage[colorlinks,citecolor=blue,urlcolor=blue]{hyperref}

\usepackage{authblk}
\usepackage{epsfig}

\usepackage{url}
\usepackage{graphicx}
\usepackage{amssymb}
\usepackage{amscd}
\usepackage{comment}
\usepackage{stmaryrd}
\usepackage{xcolor}
\usepackage{enumerate}
\usepackage{array}
\usepackage{thmtools}
\usepackage[linesnumbered,ruled,vlined]{algorithm2e}
\usepackage{caption}

\usepackage{subcaption}
\usepackage{float}
\usepackage{mathrsfs}
\usepackage{amsbsy}

\newcommand{\X}{{\cal X}}
\newcommand{\Y}{{\cal Y}}
\newcommand{\D}{{\mathscr{D}}}

\def\A{{\cal A}}

\def\S{{\cal S}}

\newcommand{\V}{{\cal V}}
\newcommand{\Z}{{\cal Z}}

\def\I{{\mathbb I}}

\newcommand{\cH}{{\cal H}}

\newcommand{\Rank}{\mathop{\rm Rank}}

\newcommand{\RR}{\mathbb{R}}

\newcommand{\EE}{\mathbb{E}}
\newcommand{\roc}{\rm ROC}
\newcommand{\auc}{\rm AUC}

\newcommand{\mmd}{\rm  MMD}

\newcommand{\bX}{\mathbf{X}}

\newcommand{\bY}{\mathbf{Y}}
\newcommand{\bZ}{\mathbf{Z}}
\newcommand{\bF}{F}

\newcommand{\dd}{\mathrm{d} }
\newcommand{\iid}{\textit{i.i.d.}}
\newcommand{\cdf}{\textit{c.d.f.}\;}
\newcommand{\df}{\textit{d.f.}}
\newcommand{\rv}{\textit{r.v.}}
\newcommand{\ie}{\textit{i.e.}\;}

\newcommand{\wrt}{\textit{w.r.t.}\;}

\newcommand{\eg}{\textit{e.g.}\;}

\newcommand{\resp}{\textit{resp.}\;}

\newcommand{\IFF}{\textit{iff.}\;}

\newcommand{\vc}{{\sc VC}}

\newtheorem{hyp}{\textit{Assumption}} 
\newtheorem{theorem}{Theorem}

\newtheorem{lemma}[theorem]{Lemma}
\newtheorem{proposition}[theorem]{Proposition}
\newtheorem{definition}[theorem]{Definition}
\newtheorem{remark}{Remark}
\newtheorem{example}{Example}
\numberwithin{equation}{section}
\theoremstyle{plain}

\title{A Bipartite Ranking Approach to the Two-Sample Problem}

    \author[1]{Stephan Cl\'emen\c{c}on}
    \author[2]{, Myrto Limnios  \thanks{Corresponding author. Authors in alphabetical order.}}
    \author[3]{, Nicolas Vayatis}

    \affil[1]{ \small \url{stephan.clemencon@telecom-paris.fr}\\
        Telecom Paris, LTCI, Institut Polytechnique de Paris\\
        19 place Marguerite Perey, Palaiseau, 91120, France.}
    \affil[2]{ \small \url{myli@math.ku.dk}\\
    University of Copenhagen, Department of Mathematical Sciences\\
    Universitetsparken 5, 2100 Copenhagen, Denmark.}
    \affil[2]{\small  \url{ nicolas.vayatis@ens-paris-saclay.fr}\\
    Université Paris Saclay, Université Paris Cité, ENS Paris Saclay\\
    CNRS, SSA, INSERM, Centre Borelli\\
    4 avenue des Sciences,	Gif-sur-Yvette, 91190, France.}
    
    \date{}

\begin{document}

	\maketitle

    \begin{abstract}
        The \textit{two-sample problem}, which consists in testing whether independent  samples on $\mathbb{R}^d$  are drawn from the same (unknown) distribution,  finds applications in many areas. Its study in high-dimension is the subject of much attention, especially because the information acquisition processes at work in the Big Data era often involve various sources, poorly controlled, leading to datasets possibly exhibiting a strong sampling bias.
        While classic methods relying on the computation of a discrepancy measure between the empirical distributions face the curse of dimensionality, we develop an alternative approach based on statistical learning and extending \textit{rank tests}, capable of detecting small departures from the null assumption in the univariate case when appropriately designed. Overcoming the lack of natural order on $\mathbb{R}^d$ when $d\geq 2$, it is implemented in two steps. Assigning to each of the samples a label (positive \textit{vs} negative) and dividing them into two parts, a preorder on $\mathbb{R}^d$ defined by a real-valued scoring function is learned by means of a bipartite ranking algorithm applied to the first part and a rank test is applied next to the scores of the remaining observations to detect possible differences in distribution. Because it learns how to project the data onto the real line nearly like (any monotone transform of) the likelihood ratio between the original multivariate distributions would do, the approach is not much affected by the dimensionality, ignoring ranking model bias issues, and preserves the advantages of univariate rank tests. Nonasymptotic error bounds are proved based on recent concentration results for two-sample linear rank-processes and an experimental study shows that the approach promoted surpasses alternative methods standing as natural competitors.
        \end{abstract}

        
        
        \section{Introduction}
        The (nonparametric) statistical hypothesis testing problem referred to as the \textit{two-sample problem} is of central importance in statistics and machine-learning. Based on two independent $\iid$ samples $\{\mathbf{X}_1,\; \ldots,\; \mathbf{X}_n\}$ and  $\{\mathbf{Y}_1,\; \ldots,\; \mathbf{Y}_m\}$
         of random variables defined on the same probability space $(\Omega,\; \mathcal{F},\; \mathbb{P})$, valued in the same space $\Z$, usually $\mathbb{R}^d$ with $d\geq 1$, and drawn from  (unknown) probability distributions $G$ and $H$ respectively, with $n,\; m\geq 1$,  it consists in testing whether the distributions $G$ and $H$ are the same or not, \textit{i.e.}, in testing the null (composite) hypothesis $\mathcal{H}_0: G=H$ against the alternative $\mathcal{H}_1: G\neq H$.\\
         
        Easy to formulate, this generic problem is ubiquitous. It finds applications in many areas, clinical trials in particular, in order to determine whether the fluctuations of a collection of measurements performed over two statistical populations subject to different treatments are simply due to the sampling phenomenon or else to the effect of the treatment. It may also be used in the context of multimodal data fusion, to decide whether two datasets can be pooled or not. Indeed, selection/sampling bias issues are also a major concern in machine learning now. As recently highlighted by theoretical and empirical works (see \textit{e.g.} \cite{BCGN21} or \cite{racialfaces}), a poor control on the acquisition process of training data, even massive, can significantly jeopardize the generalization ability of the learned predictive rules.\\
        
        Various dedicated approaches have been introduced in the statistical literature, see \textit{e.g.} section 6.9 in \cite{LehmRo05} or Chapter 3.7 in \cite{vdVWell96}.  Many of them consist in computing first nonparametric estimators $\widehat{G}_n$ and $\widehat{H}_m$ of the underlying distributions (possibly smoothed versions of the 
        empirical distributions, typically), evaluating next a distance or information-theoretic pseudo-metric $\mathcal{D}(\widehat{G}_n,\widehat{H}_m)$ between the latter in order to measure their dissimilarity (\textit{e.g.} the two-sample Kolmogorov-Smirnov statistic) and rejecting $\mathcal{H}_0$ for `large' values of the statistic $\mathcal{D}(\widehat{G}_n,\widehat{H}_m)$, see \cite{BG05} or \cite{GBRSS12} for instance. Beyond computational difficulties and the necessity of identifying a proper standardization in order to make the statistic asymptotically pivotal, {\em i.e.}, its limit distribution is parameter-free (this generally requires in practice the use of resampling/bootstrap techniques), the major issue one faces when trying to implement such \textit{plug-in} procedures is related to the curse of dimensionality. Indeed, such procedures involve the consistent estimation of distributions on a feature space of possibly very large dimension $d\in \mathbb{N}^*$.\\

        The approach promoted and analyzed in this article is very different and is inspired from \textit{rank tests} in the univariate case, see \textit{e.g.} section 6.8 in \cite{LehmRo05}. Referring to $G$ and $\{\bX_1,\;\ldots,\; \bX_n\}$ as the `positive' distribution and the `positive' sample and to $H$ and $\{\bY_1,\;\ldots,\; \bY_m\}$ as the `negative' distribution and the `negative' sample, such methods consist in computing a functional of the `positive ranks', \textit{i.e.}, the ranks of the positive instances among the pooled sample of size $N=n+m$. The rationale behind the use of such statistics, called \textit{two-sample rank statistics}, for univariate two-sample problems naturally lies in the fact that they are pivotal under the null assumption (\textit{i.e.} ignoring ties, the positive ranks are uniformly distributed on $\{1,\; \ldots,\; N\}$ under $\mathcal{H}_0$), providing tests with zero bias. The notion of two-sample \textit{linear} rank statistics, namely the (possibly weighted) sum of the images of the positive ranks divided by $(N+1)$  by a score generating function $u\in (0,1)\mapsto \phi(u)$, has been extensively studied in statistics. For an appropriate choice of $\phi$, the related rank test is known to have asymptotic power with optimal properties, describing its capacity to detect alternatives close to the null assumption, see \textit{e.g.} \cite{HajSid67} or Chapter 15 in \cite{vdV98}. For instance, the popular Mann-Whitney-Wilcoxon `ranksum' statistic, widely used to test whether a distribution is stochastically larger than another one, corresponds to the case $\phi(u)=u$ and is optimal to detect asymptotically small shifts. The methodology we propose here contrasts with straightforward extensions of such techniques to the multivariate framework based on preliminary given \textit{depth functions} (see \textit{e.g.} \cite{LiuS93}). We start from the observation that, in the univariate case, two-sample rank statistics are summaries of the statistical version of the $\roc$ curve relative to the pair of probability measures $(H,\; G)$. Under $\mathcal{H}_0$, the $\roc$ curve coincides with the diagonal of the unit square $[0,1]^2$. In particular, up to an affine transform, the ranksum statistic is nothing else than the area under the empirical $\roc$ curve ($\auc$ in abbreviated form). The method uses the $2$-split trick and is implemented in two steps. It consists in learning first from half of the samples how to rank the multivariate observations in $\mathcal{Z}$ by means of a scoring function $s:\Z\rightarrow \mathbb{R}$ transporting the natural order on $\mathbb{R}$ onto the feature space $\Z$ like (any increasing transform of) the likelihood ratio $dG/dH(z)$ would do (\textit{i.e.} the larger the score $s(Z)$ of an observation $Z$ drawn either from $H$ or $G$, the likelier $Z$ is drawn from $G$ ideally). Next, one applies a rank test to the (univariate) scores of the other half of the samples. The first step of the method thus boils down to solving a bipartite ranking problem, which can be classically formulated as a $\roc/\auc$ optimization problem, see \textit{e.g.} \cite{CV09ieee} or \cite{CV10CA}. 
        The second step consists in testing whether the obtained empirical $\roc$ curve, significantly deviates from the diagonal, where the score generating function $\phi(u)$ used in the rank test determines the nature of the deviation. 
        By means of concentration results for two-sample linear rank processes (\textit{i.e.} collections of two-sample linear rank statistics) recently established in \cite{CleLimVay21}, the (unbiased) two-stage testing procedure is analyzed from a nonasymptotic perspective. For nonparametric classes of alternative hypotheses, we prove bounds for the type- I and II errors of tests based on a scoring function $s(z)$ maximizing a statistical counterpart of a bipartite ranking performance criterion (summarizing the $\roc$ curve) taking the form of a two-sample linear rank statistic. The capacity of the two-stage method proposed to detect `small' deviations from the null assumption, preserved even in very high dimension to a certain extent, is also thoroughly investigated from an empirical angle. An extensive experimental study is presented, covering a wide variety of two-sample problems and comparing the performance of the \textit{ranking-based tests} to that of alternative nonparametric methods documented in the literature. It should be highlighted that, if the concept of $\roc$ curve is widely used to evaluate the merits of any univariate test statistic (and find an appropriate trade-off between the two types of error), the approach developed in this article is the first attempt to use this notion to devise testing procedures in a general multivariate framework.
        Notice finally that a very preliminary version of the two-stage testing method, limited to $\auc$ optimization and `ranksum' test statistics, has been previously outlined in the conference paper \cite{CDVNIPS}. The present article presents ranking-based tests and analyzes their performance in a much more general framework: building on recent results established in \cite{CleLimVay21} a wider class of ranking-based tests are considered, a finite-sample analysis of their power is carried out here and more complete experimental results are presented.\\
        
        The paper is organized as follows. In section \ref{sec:back_prelim}, the main notations are set out, the statistical framework of the two-sample problem is recalled at length and traditional methods, rank tests in the univariate case in particular, are briefly reviewed. Insights into the connection between bipartite ranking and the two-sample problem are also provided, together with an explanation of the rationale behind the approach we promote.  The novel testing procedure we propose is  described and theoretical guarantees are established in section \ref{sec:stattest}. Experimental results are presented and discussed in section \ref{sec:num}, while several concluding remarks are collected in section \ref{sec:conclusion} Technical proofs and details are deferred to the Appendix section.

        \section{Background and Preliminaries}\label{sec:back_prelim}
        In this section, the main notations are introduced and the two-sample problem is formulated in a nonparametric framework. The existing methods to solve it are briefly reviewed, with a particular attention to \textit{rank tests} in the $1$ to $d$-dimensional case when $\Z\subset \RR^d$, and their interpretation through $\roc$ analysis. Concepts and results related to the bipartite ranking task, viewed as the problem of optimizing (summaries of) the $\roc$ curve are also recalled, insofar as the methodology proposed and analyzed in the subsequent section is based on the latter.
        Here and throughout, the indicator function of any event $\mathcal{E}$ is denoted by $\mathbb{I}\{ \mathcal{E} \}$, the Dirac mass at any point $x$ by $\delta_x$, the generalized inverse of any cumulative distribution function $W(t)$ on $\RR\cup\{+\infty\}$ by $W^{-1}(u)=\inf\{t\in (-\infty,\;+\infty]:\; W(t)\geq u\}$, $u\in [0,1]$. We denote the floor and ceiling functions by $u \in \mathbb{R}\mapsto \lfloor u \rfloor$ and by $u \in \mathbb{R}\mapsto \lceil u \rceil$ respectively. For any bounded function $\psi:(0,1)\rightarrow \mathbb{R}$, we also set $\vert\vert \psi\vert\vert_{\infty}=\sup_{u\in (0,1)}\vert \psi(u)\vert$.
        Throughout the paper, bold symbols refer to multivariate variables, \textit{e.g.}, we write $X$ and $Y$ when considering univariate random variables, while $\bX$ and $\bY$ are used when the space which they take their values in the multivariate measurable space $\Z$, usually subset of $\mathbb{R}^d$ with $d\geq 2$ possibly.
        
        \subsection{The Two-Sample Problem - Nonparametric Formulation}\label{subsec:back_SoA}
        Let  $\{\bX_1,\;\ldots,\; \bX_n\}$ and  $\{\bY_1,\;\ldots,\; \bY_m\}$ be independent $\iid$ samples drawn from probability distributions $H$ and $G$ on a measurable space $\mathcal{Z}\subset \mathbb{R}^d$. In the most general version of the two-sample problem, one makes no assumption about the distributions $H$ and $G$ and the goal pursued is to test the composite hypothesis:
        \begin{equation}\label{eq:2_sample_pb}
        \mathcal{H}_0:\; G=H \text{ against the alternative } \mathcal{H}_1:\; G\neq H~,
        \end{equation}
        based on the two samples. A classic approach is to consider a probability (pseudo-) metric $\mathcal{D}$ on the space of probability distributions on $\mathcal{Z}$. Based on the simple observation that $\mathcal{D}(G,H)=0$ under the null hypothesis, a natural testing procedure consists in computing estimates of the distributions $G$ and $H$, typically the empirical distributions:  
        
        \begin{equation*}
            \widehat{G}_n=\frac{1}{n}\sum_{i=1}^n\delta_{\bX_i} \text{ and } \widehat{H}_m=\frac{1}{n}\sum_{j=1}^m\delta_{\bY_j}~,
        \end{equation*}
        and rejecting $\mathcal{H}_0$ for `large' values of the statistic $\mathcal{D}(\widehat{G}_n,$ $ \widehat{H}_m)$, see  \cite{BG05} for instance. Various metrics or pseudo-metrics can be considered for measuring dissimilarity between two probability distributions. We refer to \cite{Ra91} for an excellent account of metrics in spaces of probability measures and their applications. Typical examples include the chi-square distance, the Kullback-Leibler divergence, the Hellinger distance, the Kolmogorov-Smirnov distance and its generalizations of the following type:
        \begin{equation} \label{mmd}
        \mmd(G,H)=\sup_{f\in \mathcal{F}}\left \vert \int_{x\in \mathcal{Z}} f(x)G(\dd x)-\int_{x\in \mathcal{Z}} f(x)H(\dd x) \right \vert
        \end{equation}
        where $\mathcal{F}$ denotes a supposedly rich enough class of functions 
        $f:\mathcal{Z}\rightarrow \mathbb{R}$, so that $\mmd(G,H)=0$ if and only if $G=H$, see Th. 5 in \cite{GBRSS12}. In the univariate case, when $\mathcal{F}$ is the set of half lines, the quantity \eqref{mmd}, refered to as \textit{Maximum Mean Discrepancy} if $\mathcal{F}$ is a particular \textit{Reproducing Kernel Hilbert Space} (RKHS), naturally reduces to the usual Kolmogorov-Smirnov statistic. Beyond computational difficulties and the necessity of identifying a proper standardization in order to make the test statistic \eqref{mmd} asymptotically pivotal ({\em i.e.} its limit distribution is parameter-free), and determining an appropriate critical threshold (this generally requires in practice the use of bootstrap techniques), the major issue one faces when trying to implement such \textit{plug-in} procedures is related to the curse of dimensionality. Indeed, such procedures involve the consistent estimation of distributions on a feature space of possibly very large dimension $d\in \mathbb{N}^*$. This difficulty can however be circumvented to a certain extent when a unit ball of a 
        RKHS $\mathcal{H}$ is chosen for $\mathcal{F}$ in order to allow for efficient computation of the supremum \eqref{mmd}, see \cite{GBRSS07} and \cite{HBM08}. Lastly, a related problem in computer science literature refers to the two-sample problem as property testing, see for instance \cite{RSprop96,GGR98}. 
        The methodology promoted in the present  paper for testing \eqref{eq:2_sample_pb} 
         is very different in nature and is inspired from traditional techniques based on the notion of rank statistics in the particular one-dimensional case recalled below for clarity.

        \subsection{The Univariate Case - Rank Tests and $\roc$ Analysis}\label{subsec:back_unirankroc}
        A classic approach to the two-sample problem in the one-dimensional setup consists in ranking the observed data using the natural order on $\mathbb{R}$, and taking the decision depending on the ranks of the positive instances among the pooled sample:
        \begin{equation*}
            \forall i\in\{1,\; \ldots,\; n  \},\;\; \Rank(X_i)=N\widehat{F}_{N}(X_i)~,
        \end{equation*}
        where $\widehat{F}_{N}(t)=(1/N)(\sum_{i=1}^n\mathbb{I}\{X_i \leq t  \}+\sum_{j=1}^m\mathbb{I}\{Y_j \leq t  \})$ and $N=n+m$.
         Assuming that the distributions $G$ and $H$ are continuous for simplicity (the ties occur with probability zero), the idea underlying rank tests lies in the simple fact that, under the null hypothesis $\mathcal{H}_0$, the ranks of positive instances are distribution-free, uniformly distributed over $\{1,\; \ldots,\; N\}$. A popular choice is to consider the sum of `positive ranks', leading to the well-known \textit{rank-sum} Mann-Whitney-Wilcoxon statistic, see \cite{Wil45},
        \begin{equation}\label{eq:Wilcox}
        \widehat{W}_{n,m}=\sum_{i=1}^n \Rank(X_i)~.
        \end{equation}
        It is widely used to test $\mathcal{H}_0$ against the alternative stipulating that one of the two distributions is stochastically larger than the other one. In the situation where $G(\dd t)$ is stochastically larger\footnote{Given two distribution functions $H(\dd t)$ and $G(\dd t)$ on $\mathbb{R}\cup\{+\infty\}$, it is said that $G(\dd t)$ is \textit{stochastically larger} than $H(\dd t)$ if and only if for any $t\in\mathbb{R}$, we have $G(t)\leq H(t)$. We then write: $H\leq_{sto}G$. Classically, a necessary and sufficient condition for $G$ to be stochastically larger than $H$ is the existence of a coupling $(X,\; Y)$ of $(G,H)$, \textit{i.e.} a pair of random variables defined on the same probability space with first and second marginals equal to $H$ and $G$ respectively, such that $X\leq Y$ with probability one.} than $H(\dd t)$, \textit{i.e.}, when $H(t)\geq G(t)$ for all $t\in \mathbb{R}$, the test statistic \eqref{eq:Wilcox} is expected to take `large' values. Tables for the distribution of the statistics $\widehat{W}_{n,m}$ under $\mathcal{H}_0$ being available (even in the case where some observations are tied, by assigning the mean rank to ties, see \cite{CK97}), no asymptotic approximation result is thus needed for building a test at an appropriate level. In the case where the two \cdf are linked by the relationship $G(t)=H(t-\theta)$ with $\theta\geq 0$ ($\ie$ 
        when the treatment effect is modeled as additive such that one tests $\mathcal{H}_0:\; \theta=0$), the test statistic \eqref{eq:Wilcox}  is asymptotically uniformly most powerful in the limit experiment $\theta\searrow 0$, see \textit{e.g.} Corollary 15.14 in section 15.5 of \cite{vdV98}. Other functionals of the `positive ranks' can be used as test statistics for the two-sample problem. In particular, the class of two-sample linear rank statistics defined below forms a rich collection of functionals.
        \begin{definition}\label{def:R_stat}{\sc (Two-sample linear rank sta-tistics)} Let $\phi:[0,1]\rightarrow [0,1]$ be a nondecreasing function. The two-sample linear rank statistics with score-generating function' $\phi(u)$ based on the random samples $\{X_1,\;\ldots,\; X_n\}$ and  $\{Y_1,\;\ldots,\; Y_m\}$ is given by
        \begin{equation}\label{eq:R_stat}
        \widehat{W}^{\phi}_{n,m}=\sum_{i=1}^n\phi\left( \frac{\Rank(X_i)}{N+1} \right)~.
        \end{equation}
        \end{definition}
        For $\phi(u)=u$, the statistic \eqref{eq:R_stat} coincides with $\widehat{W}_{n,m}/(N+1)$.
        Under $\mathcal{H}_0$, the statistics \eqref{eq:R_stat} are all distribution-free, which makes them particularly useful to detect differences between the distributions $G$ and $H$. 
        They can be used to design unbiased tests at chosen levels $\alpha \in (0,1)$ by tabulating their  distribution under the null assumption. The choice of the score-generating function $\phi$ can be guided by the type of difference between the two distributions (\textit{e.g.} in scale, in location) one possibly expects, and may then lead to locally most powerful testing procedures, capable of detecting certain types of `small' deviations from $\mathcal{H}_0$. Refer to Chapter 9 in \cite{Ser80} or to Chapter 13 in \cite{vdV98} for an account of the asymptotic theory of rank statistics. For instance, choosing $\phi(u)=u\times \mathbb{I}\{u\geq u_0\}$ for $u_0\in(0,1)$ or $\phi(u)=u^q$ with $q>1$ 
        enhances the role played by the highest ranks, see \cite{CV07} or \cite{Rud06}. We also emphasize that concentration properties of two-sample linear rank processes (\textit{i.e.} collections of two-sample linear rank statistics) have recently been studied in \cite{CleLimVay21}, motivated by the interpretation of \eqref{eq:R_stat} as a scalar statistical summary of the $\roc$ curve relative to the pair $(H,G)$. Based on these results, an exponential tail probability bound for \eqref{eq:R_stat} is proved in the Appendix section (see Theorem \ref{thm:tail} therein), which is used in the theoretical analysis carried out in section \ref{sec:stattest}.
        \medskip
        
        \noindent {\bf $\roc$ analysis.} The $\roc$ curve is a gold standard tool to describe the dissimilarity between two univariate probability distributions $G$ and $H$. This criterion of functional nature,  $\roc_{H,G}$, can be defined as the Probability-Probability plot:
        $ t\in \mathbb{R}\mapsto \left( 1-H(t),\; 1-G(t)  \right)$,
        where possible jumps are connected by line segments, ensuring that the resulting curve is continuous. With this convention, one may then see the $\roc$ curve related to the pair of $\df$ $(H,G)$ as the graph of a c\`ad-l\`ag (\textit{i.e.} right-continuous and left-limited) non decreasing mapping valued in $[0,1]$, defined by $\alpha \in (0,1) \mapsto  1-G\circ H^{-1}(1-\alpha)$
        at points $\alpha$ such that $G\circ H^{-1}(1-\alpha)=1-\alpha$. Denoting by $\Z_H$ and $\Z_G$ the supports of $H$ and $G$ respectively, observe that it connects the point $(0,1-G( \Z_H))$ to $(H(\Z_G),1)$ in the unit square $[0,1]^2$ and that, in absence of plateau 
        (which we assume here for simplicity, rather than restricting the feature space to $G$'s support), the curve $\alpha\in (0,1)\mapsto \roc_{G,H}(\alpha)$ is the image of $\alpha\in (0,1)\mapsto \roc_{H,G}(\alpha)$ by the reflection with the main diagonal of the Euclidean plane (\textit{i.e.} the line of equation `$\beta=\alpha$') as axis. Notice that the curve $\roc_{H,G}$ coincides with the main diagonal of $[0,1]^2$ \IFF $H=G$, $\ie$, $\mathcal{H}_0$ is true. Hence, the concept of $\roc$ curve offers a visual tool to examine the differences between two distributions in a pivotal manner, see Fig. \ref{fig:ex1}. For instance, the univariate distribution $G(dt)$ is stochastically larger than $H(dt)$ \IFF the curve $\roc_{H,G}$ is everywhere above the main diagonal and $\roc_{H,G}$ coincides with the left upper corner of the unit square \IFF the essential supremum of the distribution $H$ is smaller than the essential infimum of the distribution $G$. 
        \begin{figure}[ht!]
        \centering
        \begin{tabular}{cc}
        \parbox{6cm}{
        \begin{center}
        \includegraphics[width=0.3\textwidth]{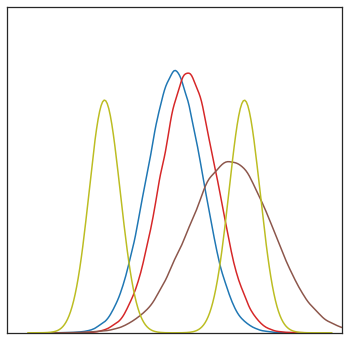}\\
        \small{a. Probability distributions}
        \end{center}
        }
        \parbox{6cm}{
        \begin{center}
        \includegraphics[width=0.3\textwidth]{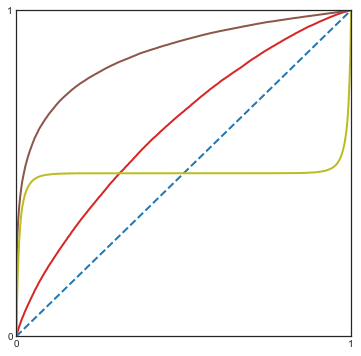}\\
        \small{b. $\roc$ curves}
        \end{center}
        }
        \end{tabular}
        \caption{Examples of pairs of distributions and their related $\roc$ curves. The `negative' distribution $H$ is represented in blue and three typical examples of `positive' distributions with their associated $\roc$ curves are in red (location model), brown (location and scale model) and green (mixture model).}
        \label{fig:ex1}
        \end{figure}
        
        Based on the samples $\{X_1,\;\ldots,\; X_n\}$ and  $\{Y_1,\;\ldots,\; Y_m\}$, a statistical version of $\roc_{H,G}$ is obtained by computing the empirical $\cdf$ 
        $\widehat{H}_m(t)=(1/m)\sum_{j=1}^m\mathbb{I}\{Y_j\leq t\}$ and $\widehat{G}_n(t)=(1/n)\sum_{i=1}^n\mathbb{I}\{X_i\leq t\}$ for $t\in \mathbb{R}$, and plotting the empirical $\roc$ curve:
        \begin{equation}\label{eq:emp_ROC}
        \widehat{\roc}_{\text{H,G}}=\roc_{\widehat{H}_m, \; \widehat{G}_n}~.
        \end{equation}
         Observe that the $\roc$ curve \eqref{eq:emp_ROC}
         is fully determined by the set of ranks occupied by the positive instances within the pooled sample $\{ \Rank(X_i):\; i=1,\; \ldots,\; n\}$.
         
        Breakpoints of the piecewise linear curve \eqref{eq:emp_ROC} necessarily belong to the set of gridpoints:
        \begin{equation*}
            \left\{ \left(j/m,\; i/n  \right):\; j\in\{1,\; \ldots,\; m-1  \},\;  i\in\{1,\; \ldots,\; n-1  \}  \right\}~.
        \end{equation*} 
        Denote by $X_{(i)}$ the order statistics related to the sample $\{X_1,\;\ldots,\; X_n\}$, \textit{i.e.},
        $
        \Rank(X_{(n)})>\cdots >\Rank(X_{(1)})
        $,
        the $\roc$ curve \eqref{eq:emp_ROC} is the continuous broken line that connects the jump points of the step curve:
        \begin{equation}\label{eq:step}
        \alpha\in [0,1]\mapsto \sum_{j=1}^m \widehat{\gamma}_j \times \mathbb{I}\{\alpha\in [(j-1)/m,\; j/m[\}~,
        \end{equation}
        where, for all $j\in\{1,\; \ldots,\; m\}$, we set
        \begin{equation*}
        \widehat{\gamma}_j= \frac{1}{n}\sum_{i=1}^n\mathbb{I}\{j\geq N-\Rank(X_{(n-i+1)})-i+2  \}~.
        \end{equation*}
        The empirical $\roc$ curve \eqref{eq:emp_ROC} is expressed as a function of the $\Rank(X_i)$'s only.
        As a consequence, any of its summary is a two-sample rank statistic, that is a measurable function of the `positive ranks'. In particular, choosing 
        the empirical counterpart of the Area Under the $\roc$ Curve ($\auc$ in short) is the popular $\roc$ curve summary, given by
        \begin{equation}\label{eq:auc}
        \auc_{H,G}:=\int_{0}^1\roc_{H,G}(\alpha)\dd \alpha
        =\mathbb{P}\left\{Y<X \right\}+\frac{1}{2}\mathbb{P}\left\{X=Y \right\}~.
        \end{equation}
        
        Indeed, the $\auc$ of the empirical $\roc$ curve \eqref{eq:emp_ROC}, also termed the rate of concording pairs, can be easily shown to coincide with the  \textit{ranksum} Mann-Whitney-Wilcoxon statistic \eqref{eq:Wilcox} up to an affine transformation
        \begin{equation*}
        \widehat{W}_{n,m}=nm\auc_{\widehat{H}_m, \widehat{G}_n}+\frac{n(n+1)}{2}~.
        \end{equation*}
        More generally, two-sample linear rank statistics \eqref{eq:R_stat} related to score generating functions different from $\phi(u)=u$ provide alternative summaries of the empirical $\roc$ curve and measure different ways of deviating from the main diagonal of the unit square, which coincides with $\roc_{H,H}$ under the null assumption. Notice incidentally that, under $\mathcal{H}_0$, we have
        \begin{equation*}
            \auc_{H,G}-1/2=\int_{\alpha=0}^1\left\{\roc_{H,G}(\alpha)-\alpha\right\}\dd \alpha=0~.
        \end{equation*} 
        Like \eqref{eq:Wilcox}, the statistics  \eqref{eq:R_stat} are pivotal and in order to quantify their fluctuations as the full sample sizes $n,\; m$ increase, the fraction of `positive'/`negative' observations in the pooled dataset must be controlled. Let $p\in (0,1)$ be the `theoretical' fraction of positive instances. For $N\geq 1/p$, we suppose that $n = \lfloor pN \rfloor$ and $m = \lceil (1-p)N \rceil = N - n$. Define the mixture probability distribution $F=pG+(1-p)H$. As $N$ tends to infinity, the asymptotic mean of  $\widehat{W}^{\phi}_{n,m}/n$ is:
        \begin{multline}\label{eq:lim_mean}
        W_{\phi}(H,G)=\mathbb{E}[\phi\circ F(X)] =\frac{1}{p}\int_{0}^1\phi(u) \dd u \\
        -\frac{1-p}{p}\int_{0}^1 \phi\left(p(1-\roc_{H,G}(\alpha))+(1-p)(1-\alpha)\right) \dd \alpha~, 
        \end{multline}
        see section 3 in \cite{CleLimVay21}. In absence of any ambiguity about the pair $(H,G)$ of univariate distributions considered, we write $W_{\phi}$ for the sake of simplicity. Observe that, under $\mathcal{H}_0$, we have
        \begin{equation*}
            W_{\phi}=\int_{0}^1\phi(u) \dd u~.
        \end{equation*}

        \noindent {\bf Building a two-sample test in the $\roc$ space.} 
        Under $\mathcal{H}_0$, the theoretical $\roc$ curve coincides with the main diagonal of $[0,1]^2$: $\roc_{H,H}(\alpha)=\alpha$, for all $\alpha\in (0,1)$ and any distribution $H$ on $\mathbb{R}$. In addition, since the empirical $\roc$ curve is itself a function of the ranks, it is also a (functional) pivotal statistic under the null assumption. When the probability distribution $H$ is continuous,  all the possible empirical $\roc$ curves are then equiprobable under $\cH_0$. 
        Hence, in the situation where $\binom{N}{n}$ is not too large, since the ensemble $\mathcal{C}_{n,m}$ of all possible empirical $\roc$ curves based on positive and negative samples of respective sizes $n$ and $m$ is of cardinality $\binom{N}{n}$,  all broken lines included in it can be enumerated (see Fig. \ref{fig:emprocmc}) and for any $\alpha\in \left\{i/\binom{N}{n}:\; i=1,\; \ldots,\; \binom{N}{n} \right\}$, one can build a tolerance/prediction region $\mathcal{R}_{\alpha}\subset \mathcal{C}_{n,m}$ of level $\alpha$, \textit{i.e.}, a subset $\mathcal{R}_{\alpha}\subset \mathcal{C}_{n,m}$ of cardinality $\alpha\binom{N}{n}$. Then, a test rejecting $\mathcal{H}_0$ when the empirical $\roc$ curve $\widehat{\roc}_{H,G}$ falls outside $\mathcal{R}_{\alpha}$ can be considered.
        A natural way of building a critical region in the space $\mathbb{D}([0,1])$ of c\`ad-l\`ag mappings  $[0,1]\to[0,1]$,  defining a test of hypothesis $\mathcal{H}_0$ at level $\alpha$, is to: fix a pseudo-distance $D$ on $\mathbb{D}([0,1])$, sort the  $\binom{N}{n}$ curves in $\mathcal{C}_{n,m}$ by increasing distance to the first diagonal, and keep the subset $\mathcal{R}_{\alpha}$ formed by the $\lceil\binom{N}{n}(1-\alpha)\rceil$ curves closest to the diagonal in the sense of the chosen distance $D$. When choosing the distance defined by $L_1$-norm, one naturally recovers the Mann-Whitney-Wilcoxon test. However, many functional distances can be considered for this purpose.

        \begin{remark}\label{rk:AUC}  {\sc (The $\auc$ as a statistical distance)} One may easily show that
        \begin{equation}\label{eq:AUC_connect}
        \auc_{H,G}=\frac{1}{2}+\int_{-\infty}^{\infty}\{H(t)-G(t)\}\dd H(t)~,
        \end{equation}
        see the Appendix section for further details. Hence, when $H$ is stochastically smaller than $G$, the quantity $\auc_{H,G}-1/2$ is equal to the $L_1(H)$-distance between the $\cdf$ $H(t)$ and $G(t)$.
        \end{remark}

        \begin{figure}[ht!]
            \centering
          
                \begin{tabular}{cc}
                \parbox{6cm}{
                \begin{center}
                \includegraphics[width=0.3\textwidth]{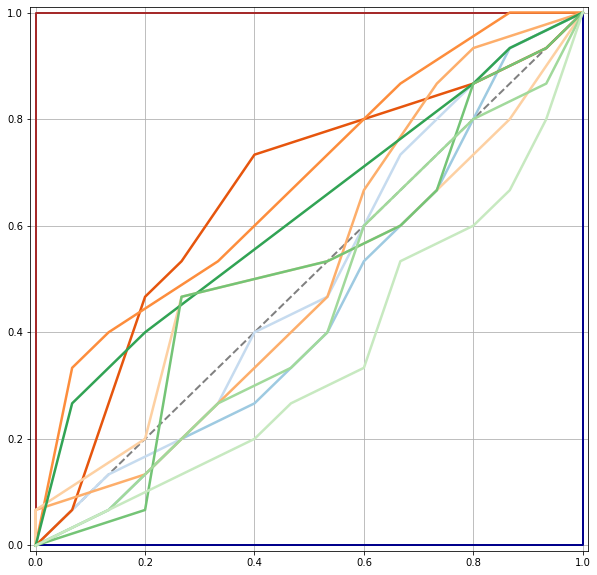}\\
                        \small{a. Empirical $\roc$ curves with $n = m = 15$ }
                    \end{center}
                }
                \parbox{6cm}{
                    \begin{center}
                        \includegraphics[width=0.3\textwidth]{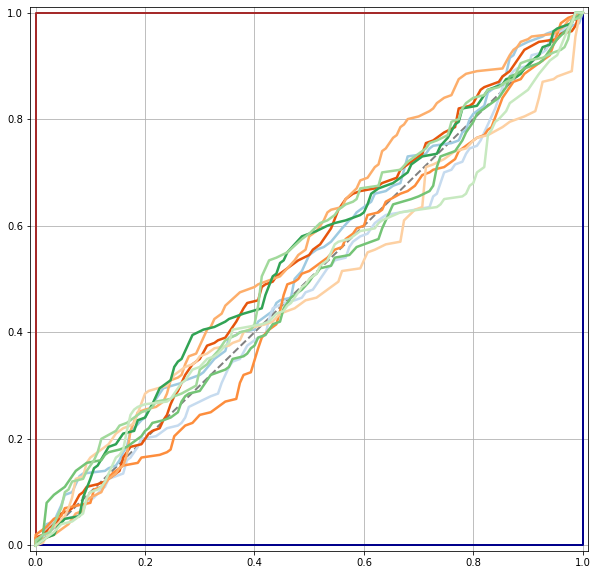}\\
                        \small{b.  Empirical $\roc$ curves with $n, m = 200, 150$}
                    \end{center}
                }
            \end{tabular}
            \caption{Examples of empirical $\roc$ curves simulations under the null hypothesis.} 
            \label{fig:emprocmc}
        \end{figure}
        
        \noindent{\bf  Extensions of two-sample rank tests to the multivariate framework.}  Given the absence of any `natural order' on $\mathbb{R}^d$ as soon as $d\geq2$,  existing methods explored different concepts of multivariate ranks statistics.
        Beyond the approaches based on component-wise ranks and copula (see \textit{e.g.} \cite{PuriSen71} or \cite{LunLevCa12}), provably valid under strong assumptions on the probabilistic model, the concept of \textit{statistical depth} aims at defining a center-outward ordering of the points in the support of a multivariate distribution $P(dx)$ on $\mathbb{R}^d$ in order to emulate ranks, see \cite{Mosler13}.
        Precisely,  a \textit{depth function} relative to $P$ is a bounded non-negative Borel-measurable mapping $D_P:\mathbb{R}^d\rightarrow \mathbb{R}_+$ that defines a preorder for multivariate points in $\mathbb{R}^d$ and hopefully determines the centrality of any point $x\in \mathbb{R}^d$ with respect to the probability measure $P$. Thus, 
        points $x\in \mathbb{R}^d$ near the `center' of the mass are the deepest, \textit{i.e.},  
        $D_P(x)$ is among the highest values taken by the depth function. Originally introduced in the seminal contribution \cite{Tukey75}, the \textit{half-space depth} of $x$ in $\mathbb{R}^d$ relative to $P$, is the minimum of the mass $P(S)$ taken over all closed half-spaces $S\subset \mathbb{R}^d$ such that $x\in S$. Many alternatives have been developed, refer to \textit{e.g.} \cite{Liu}, \cite{koshevoy1997}, \cite{Chaudhuri}, \cite{OJA1983}, \cite{Vardi1423}, \cite{chernozhukov2017}, \cite{DebSen19} or  \cite{BeiHall20}. In \cite{ZuoSerfling00}, an axiomatic nomenclature of statistical depths has been devised, providing a systematic way of comparing their merits and drawbacks. As proposed by \cite{LiuS93}, assuming that a certain notion of depth has been selected, a natural way of extending two-sample rank tests to the multivariate framework consists in considering the largest of the two samples available, the $\mathbf{X}$ sample say, using next part of it to compute the sampling version $D_{\widehat{S}}$ of the depth relative to $S$. Then, apply a univariate two-sample rank test to the (univariate) sample formed by the depth values of the $\mathbf{X}_i$'s, that have not been involved in the depth estimation step and that formed by the depth values of the $\mathbf{Y}_j$'s. The main limitation of depth-based rank tests naturally lies in the impact of the notion of depth considered: different choices highlight different features of the distributions, leading to possibly different decisions.
        Refer  to \cite{HalBaCAMa18}, subsection 2.1 and Appendix A.2, for a comprehensive understanding of homogeneity tests based on center-outward distributions and review on  the vast 
        literature dedicated to the concept of spatial ranks (and signs), see \eg \cite{OjaMotto95,OjaMotto97,OjaMott05}. Such models are usually held for testing classic parametric alternatives to the homogeneity hypothesis $\mathcal{H}_0$ (\textit{e.g.} location, scale). We refer to \cite{Oja10} for a comprehensive review of these approaches and to \cite{ChakChau15,ChakChau17} when considering high-dimensional settings. 
        For the sake of completeness, we finally point out the works devoted to (semi-)parametric ranks based on the Mahalanobis statistical distance, applicable to the class of elliptical distributions only, see \eg \cite{UmRan98}, \cite{HallPain02AN, HallPain02bern, HallPai08scatter}. \\
        
        As shall be seen, the approach sketched in the next subsection, and analyzed at length in section \ref{sec:stattest}, shares some similarities with the method relying on statistical depth, except that the mapping used to `project' the multivariate observations onto the real line is specifically learned from the data in order to detect at best the deviations in distribution between the two samples.
        The following subsection precisely explains how any scoring function $s:\mathcal{Z}\rightarrow \mathbb{R}$, solution of the bipartite ranking problem related to the pair $(H,G)$, permits to extend the use of $\roc$ analysis and two-sample rank statistics to the two-sample problem in the multivariate setup.
        
        \subsection{Bipartite Ranking - The Rationale Behind our Approach}\label{subsec:rationale}
        The goal of bipartite ranking is to learn, based on the `positive' and `negative' samples $\{\bX_1,\;\ldots,\; \bX_n\}$  and  $\{\bY_1,\;\ldots,\; \bY_m\}$, how to score any new observations $\bZ_1,\; \ldots,\; \bZ_k$, being each either `positive' or else `negative', that is to say drawn either from $G$ or else from $H$, without prior knowledge, so that positive instances are mostly at the top of the resulting list with large probability. A natural way of defining a total preorder\footnote{A preorder $\preccurlyeq$ on a set $\mathcal{Z}$ is a reflexive and transitive binary relation on $\mathcal{Z}$. It is said to be \textit{total}, when either $z\preccurlyeq z'$ or else $z'\preccurlyeq z$ holds true, for all $(z,z')\in\mathcal{Z}^2$.} on $\mathcal{Z}$ is to map it with the natural order on $\mathbb{R}\cup \{+\infty \}$ by means of a \textit{scoring rule}, \textit{i.e.}, a measurable mapping $s:\mathcal{Z}\rightarrow (-\infty,\; \infty]$. A preorder $\preccurlyeq_s$ on $\mathcal{Z}$ is then defined by: for all $(x,x') \in \mathcal{Z}$, $x\preccurlyeq_s x'$ \IFF $s(x)\leq s(x')$. We denote by $\S$ the set of all scoring functions. The capacity of a candidate $s(z)$ in $\S$ to discriminate between the positive and negative statistical populations is generally evaluated by means of the $\roc$ curve $\roc_{H_s,G_s}(\alpha)=\roc(s,\alpha)$, where $H_s$ and $G_s$ the pushforward probability distributions of $H$ and $G$ by the mapping $s(z)$. 
        It offers a visual tool for assessing ranking performance: the closer to the left upper corner of the unit square the curve $\roc(s,.)$, the better the scoring rule $s$. Therefore, the $\roc$ curve  conveys a partial preorder on the set of all scoring functions: for all pairs of scoring functions $s_1$ and $s_2$, one says that $s_2$ is more accurate than $s_1$ when $\roc(s_1,\alpha)\leq \roc(s_2,\alpha)$ for all $\alpha\in [0,1]$. It follows from a standard Neyman-Pearson argument that the most accurate scoring rules are increasing transforms of the likelihood ratio $\Psi(z)=dG/dH(z)$. Precisely, \cite{CV09ieee},  Proposition 4 therein, proved that the optimal set of elements is 
        \begin{equation}
        \S^*=\left\{s\in \S, \; \;\forall (z,\; z' ) \in\mathcal{Z}^2,\;\; \Psi(z)<\Psi(z') \Rightarrow s^*(z)<s^*(z')    \right\}~.
        \end{equation}
        And, for all $(s,\; \alpha)\in \S\times (0,1)$: 
        \begin{equation*}
            \roc(s,\; \alpha)\leq \roc^*(\alpha)~,
        \end{equation*}
        
        \noindent where $\roc^*(\cdot)=\roc(\Psi,\; \cdot)=\roc(s^*,\; \cdot)$ for any $s^* \in \S$. Recall 
         that this optimal curve is non-decreasing and concave, thus always above the main diagonal of the unit square (more properties are recalled in Appendix \ref{app:rocprop}). Now, the bipartite ranking task can be reformulated in a more quantitative manner: the objective pursued is to build a scoring function $s(z)$, based on the training random samples $\{\bX_1,\;\ldots,\; \bX_n\} \text{ and } \{\bY_1,\;\ldots,\; \bY_m\}$, with a $\roc$ curve as close as possible to $\roc^*$. A typical 
         measure for the deviation between the two curves is to consider the distance in $\sup$ norm:
        \begin{equation}\label{eq:sup_dist}
        d_{\infty}(s, s^*)=\sup_{\alpha\in (0,1)}\left\vert \roc(s,\alpha)-\roc^*(\alpha) \right\vert~.
        \end{equation}
        This quantity is a distance between $\roc$ curves (or between the related equivalence classes of scoring functions, the $\roc$ curve of any scoring function being invariant by strictly increasing transform) not between the scoring functions themselves. Since the curve $\roc^*$ is unknown in practice, the major difficulty is that 
        no straightforward statistical counterpart of the (functional) loss \eqref{eq:sup_dist} is available. 
        \cite{CV09ieee}
         (see also \cite{CV10CA}) 
         proved that bipartite ranking can be viewed as a superposition of cost-sensitive classification problems and 
          `discretized' in an adaptive manner, 
          thus applying empirical risk minimization with statistical guarantees in the $d_{\infty}$-sense. The price of that procedure is an additional bias term inherent to the approximation step.
         Alternatively, the performance of a candidate scoring rule $s$ can be measured by means of the $L_1$-norm in the $\roc$ space. Observing that, in this case, the loss can be decomposed as follows:
         \begin{multline}\label{eq:L1_norm}
         d_1(s,s^*)=\int_{0}^1\vert \roc(s,\alpha)-\roc^*(\alpha)  \vert \dd \alpha =\int_{0}^1\roc^*(\alpha)  d\alpha -\int_{0}^1 \roc(s,\alpha) \dd \alpha~,
         \end{multline}
         minimizing the $L_1$-distance to the optimal $\roc$ curve boils down to maximizing the area under the curve $\roc(s,\; \cdot)$:
        \begin{equation*}
        \auc(s):=\auc_{H_s,G_s}=\mathbb{P}\{s(\bY)<s(\bX) \}+\frac{1}{2}\mathbb{P}\{s(\bY)=s(\bX) \}~,
        \end{equation*}
        where $\bX$ and $\bY$ are random variables defined on the same probability space, independent, with respective distributions $G$ and $H$, denoting by $G_s$ and $H_s$ the distributions of $s(\bX)$ and $s(\bY)$ respectively.
        The scalar performance criterion $\auc(s)$ defines a total preorder on $\mathcal{S}$ and its maximal value is denoted by $\auc^*=\auc(s^*)$, with $s^*\in \S^*$. Bipartite ranking through maximization of empirical versions of the $\auc$ criterion has been studied in several articles, including \cite{AGHHPR05} or \cite{CLV08}. Considering a class $\S_0\subset\S$, upper confidence bounds for the $\auc$ deficit of scoring rules obtained by solving the problem:
        \begin{equation}
        \max_{s\in \S_0} \; \auc_{\widehat{H}_{s,m},\widehat{G}_{s,n}}~,
        \end{equation}
        where  $\widehat{H}_{s,m}=(1/m)\sum_{j=1}^m\delta_{s(\bY_j)}$ and  $\widehat{G}_{s,n}=(1/n)\sum_{i=1}^n\delta_{s(\bX_i)}$, have been established in particular. 
        Notice that maximizing the empirical version of $\auc(s)$ over a class $\S_0$ of scoring rule candidates boils down to maximizing the rank-sum criterion:
        \begin{equation*}
        \widehat{W}_{n,m}(s):=\sum_{i=1}^n\Rank(s(\bX_i))~,
        \end{equation*}

        where $\Rank(s(\bX_i)) =N\widehat{F}_{s,N}(s(\bX_i))$  for $i\in\{1,\; \ldots,\; n  \}$ and $\widehat{F}_{s,N}(t)=(1/N)\sum_{i=1}^n\mathbb{I}\{s(\bX_i) \leq t  \}+(1/N)\sum_{j=1}^m\mathbb{I}\{s(\bY_j) \leq t  \}$ for $t\in \mathbb{R}$.
         In \cite{CleLimVay21} (see also \cite{CV08NIPS1}), the performance of  scoring rules $s(z)$ maximizing alternative scalar criteria, of the form of two-sample linear rank statistics based on the univariate samples $\{s(\bX_1),\;\ldots,\; s(\bX_n)\}$ and $\{s(\bY_1),\;\ldots,\; s(\bY_m)\}$ as well, has been investigated. Precisely, considering a non-decreasing score-generating function $\phi(u)$, the $W_{\phi}$-ranking performance criterion is defined as:
        \begin{multline}\label{eq:W_crit}
        W_{\phi}(s)=\mathbb{E}\left[ (\phi \circ F_s)(s(\bX)) \right]
         = \frac{1}{p}\int_{0}^1\phi(u)\dd u \\
        -\frac{1-p}{p}\int_{0}^1 \phi\left(p(1-\roc(s,\; \alpha))+(1-p)(1-\alpha)\right)\dd \alpha~,
         \end{multline}
         where $F_s=pG_s+(1-p)H_s$ for any $s\in \S$. Equipped with this notation, we have $W_{\phi}(s)\leq W^*_{\phi}(H,G):=W_{\phi}(s^*)$ for any $s^*\in \S^*$, and when $\phi$ is strictly increasing, 
         the set of maximizers of the criterion $W_{\phi}$ coincides with $\S^*$, see Proposition 6 in  \cite{CleLimVay21}. For simplicity, we write $W^*_{\phi}(H,G)=W^*_{\phi}$, when there is no ambiguity about the pair $(H,G)$  of probability distributions considered. Whereas maximizing the quantity \eqref{eq:W_crit} boils down to performing $\auc$ maximization when $\phi(u)=u$, some specific patterns of the preorder induced by a scoring function $s(z)$ can be more or less enhanced 
         depending on the score-generating function $\phi$ chosen. For $\phi(u)=u^q$ with $q>1$, or any other score generating function $\phi$ that rapidly vanishes near $0$ and takes much higher values near $1$, for instance, the value of \eqref{eq:W_crit} is essentially determined by the behavior of the curve  $\roc(s,\; \cdot)$ near $0$ (\textit{i.e.} the probability that $s(\mathbf{X})$ takes high values), as discussed in subsection 2.3 of \cite{CleLimVay21}.
        \medskip

        \noindent{\bf Ranking-based two-sample rank tests.} The two-sample test procedures rely on the observation that deviations of the curve $\roc^*$ from the main diagonal of $[0,1]^2$, as well as those of $W^*_{\phi}$ from $\int_0^1\phi(u)\dd u$ for appropriate score generating functions $\phi$,  provide a natural way of measuring the dissimilarity beween $G$ and $H$ in theory. As revealed by the proposition below, such deviations are equal to zero as soon as the null assumption is fulfilled.
        \begin{proposition}\label{prop:rationale}
        The following assertions are equivalent.
        \begin{itemize}
        \item[(i)] The assumption `$\mathcal{H}_0:\; H=G$' holds true.
        \item[(ii)] The optimal $\roc$ curve relative to the bipartite ranking problem defined by the pair $(H,G)$ coincides with the diagonal of $[0,1]^2$
        $$
        \forall \alpha\in (0,1),\;\; \roc^*(\alpha)=\alpha~.
        $$
        \item[(iii)] For any score-generating function $\phi(u)$, we have
        $$
        W^*_{\phi}=\int_0^1\phi(u)\dd u~.
        $$
        \item[(iv)] There exists a strictly increasing score-generating function $\phi(u)$, such that:
        $$
        W^*_{\phi}=\int_0^1\phi(u)\dd u~.
        $$
        \item [(iv)] We have $\auc^*=1/2$.
        \end{itemize}
        In addition, we have:
        \begin{equation}\label{eq:aucdist}
        \auc^*-1/2=\mathbb{E}[\vert \Psi(\bY)-1 \vert  ]~.
        \end{equation}
        \end{proposition}
        We also recall that the optimal $\roc$ curve related to the pair of distributions $(H,G)$ is the same as that related to the pair of univariate distributions $(H_{s^*},G_{s^*})$ and that $\dd G/\dd H(z)=\dd G_{s^*}/\dd H_{s^*}(s^*(z))$ for any $s^*\in \S^*$, see Corollary 5 in \cite{CV09ieee}. Hence, the optimal curve $\roc^*$ is a very natural and exhaustive way of measuring the dissimilarity between two multivariate distributions, extending the basic $\roc$ analysis for distributions on $\mathbb{R}$ recalled in subsection \ref{subsec:back_unirankroc}, as illustrated by the example below.
        
        \begin{example}{\sc (Multivariate Gaussian populations)}\label{exgauss} Consider two Gaussian distributions $H$ and $G$ on $\mathbb{R}^d$ with same positive definite covariance matrix $\Gamma$ and respective means $\theta_-$ and $\theta_+$ in $\mathbb{R}^d$, supposed to be distinct. As an increasing transform of the $\log $likelihood ratio, the scoring function
        \begin{equation*}
         z\in \mathbb{R}^d\mapsto s(z)=\langle z,\; \Gamma^{-1}(\theta_+-\theta_-)\rangle
        \end{equation*}
        is optimal, denoting by $\langle .\;,\; .\rangle$ the usual Euclidean inner product on $\mathbb{R}^d$. Since it is linear, the pushforward measures $H_s$ and $G_s$ are both univariate Gaussian distributions. Denoting $\Delta(t)=(1/\sqrt{2\pi})\int_{-\infty}^t\exp(-u^2/2)\dd u$, $t\in \mathbb{R}$, the $\cdf$ of the centered standard univariate Gaussian distribution, one may immediately check that the optimal $\roc$ curve is given by, for all $\alpha\in (0,1)$,
        \begin{equation*}
        \roc^*(\alpha)= 1-\Delta\left( \Delta^{-1}(1-\alpha) - \sqrt{s( \theta_+-\theta_-)}\right) .
        \end{equation*}
        In addition, we have
        \begin{equation*}
        \auc^*-1/2= 1 - \exp\{s( \theta_+-\theta_-)\}~.
        \end{equation*}
        \end{example}
        Hence, Proposition \ref{prop:rationale} permits to reformulate the nonparametric test problem \eqref{eq:2_sample_pb} in various ways, as follows for instance:
        \begin{equation}\label{eq:2_sample_pb_alt}
        \mathcal{H}_0:\;\; \auc^*=1/2 \quad \textit{vs.} \quad \mathcal{H}_1: \auc^*>1/2~,
        \end{equation}
        or, equivalently, as:
        \begin{equation}\label{eq:2_sample_pb_alt2}
        \mathcal{H}_0:\;\; W_{\phi}^*=\int_0^1\phi(u)\dd u \quad \textit{vs.} \quad  \mathcal{H}_1: \;\; W_{\phi}^*>\int_0^1\phi(u)\dd u~,
        \end{equation}
        for any given strictly increasing score generating function $\phi(u)$.
        It is noteworthy that the formulations above are \textit{unilateral}, in contrast with the 
        classic rank-sum test in the univariate case: indeed, $G^*$ is always stochastically larger than $H^*$, whereas, given two arbitrary univariate probability distributions, one is not necessary stochastically larger than the other.

         As the optimal $\roc$ curve and its summaries, such as the quantities $\auc^*$ or $W^*_{\phi}$, are unknown in practice, we propose an approach for solving the two-sample problem  implemented in two steps. By splitting the samples $\{\bX_1,\;\ldots,\; \bX_n\}$ and $\{\bY_1,\;\ldots,\; \bY_m\}$ into two halves: 1) first, solve the bipartite ranking problem based on the first halves of the `positive' and `negative' samples producing a scoring function $\widehat{s}(z)$, as described in the preceding subsection; 2) then, perform a univariate rank-based test  on the remaining data samples mapped by the obtained scoring function $\widehat{s}(z)$.  The subsequent sections 
          provide both theoretical and empirical evidence that, beyond the fact that they are nearly unbiased, such testing procedures permit to detect very small deviations from the null assumption.

        \section{Ranking-based Rank Tests for the Two-Sample Problem}\label{sec:stattest}
        In this section, we describe at length the two-sample methodology foreshadowed by the observations made in the preceding section and  discuss its possible practical implementations. Its theoretical properties (level and power)  are next analyzed from a nonasymptotic angle under specific assumptions.
        
        \subsection{Method and Implementations}\label{subsec:method}
        We explain how the general idea sketched in sec. \ref{subsec:rationale} can be applied effectively, based on the observation of two independent $\iid$ samples $\{\mathbf{X}_1,\; \ldots,\; \mathbf{X}_n\}$ and $\{\mathbf{Y}_1,\; \ldots,\; \mathbf{Y}_m\}$ with $n,\; m\geq 1$. Let $\alpha\in (0,1)$ be the target level, \textit{i.e.}, the desired type-I error. As previously discussed, two ingredients are essentially involved in the testing procedure: 
        \begin{enumerate} 
        \item A bipartite ranking algorithm $\mathcal{A}:\mathcal{Z}^{n+m}\rightarrow \mathcal{S}_0$ operating on a class $\mathcal{S}_0\subset \S$  and assigning to any set of training observations $\mathscr{D}_{n,m}=\{\mathbf{x}_1,\; \ldots,\; \mathbf{x}_n\}\cup \{\mathbf{y}_1,\; \ldots,\; \mathbf{y}_m\}$  a scoring function $\mathcal{A}(\mathscr{D}_{n,m})$ in $\S_0$;
        \item A two-sample rank test
        $\Phi_{\alpha}^{\phi}:\;  (-\infty, +\infty]^{n+m}\rightarrow  \{0, 1\}
        $ of level $\alpha\in (0,1)$
        with outcome depending on  $\{\Rank(x_1),\; \ldots,\; \Rank(x_n)\}$ for any pooled univariate dataset 
        $\mathscr{D}_{n,m}= \{x_1,\; \ldots, \; x_n\}\cup \{ y_1,\; \ldots,\; y_m \}$ and score-generating function $\phi$.
        \end{enumerate}
        Equipped with these two components, the methodology, relying on the computation of a Ranking-based two-sample rank tests, is implemented in two main steps, as summarized in Fig. \ref{twostageproc}.

        \begin{figure}[t!]
        \fbox{\begin{minipage}{\textwidth} 
            \medskip
        \begin{center}
        {\large \textbf{Ranking-based Two-sample Rank Tests}}
        \end{center}		
        \begin{enumerate}		
            \item[]{\bf Input.} Two independent and $\iid$ samples $\{\bX_1,\; \ldots,\; \bX_n\}$ and $\{\bY_1,\; \ldots,\; \bY_m\}$ of sizes $n,\; m\geq 2$ and valued in $\mathcal{Z}$;  subsample sizes $n'<n$ and $m'<m$; bipartite ranking $\mathcal{A}$ algorithm operating on the class $\S_0$ of scoring functions defined on $\mathcal{Z}$; univariate two-sample rank test $\Phi_{\alpha}^{\phi}$ of level  $\alpha\in (0,1)$.\\

                \item[] {\bf 2-split trick.} Divide each of the original samples into two subsamples
                \begin{equation*}
            \{\bX_1,\; \ldots,\; \bX_{n'}\}	\cup \{\bX_{1+n'},\; \ldots,\; \bX_n\}
             \text{ and } \{\bY_1,\; \ldots,\; \bY_{m'}\}	\cup \{\bY_{1+m'},\; \ldots,\; \bY_m\}
                \end{equation*}
                \item {\bf Bipartite Ranking.} Run the bipartite ranking algorithm $\mathcal{A}$ based on the training data $\mathscr{D}_{n',m'}=\{\bX_1,\; \ldots,\; \bX_{n'}\}	\cup \{\bY_1,\; \ldots,\; \bY_{m'}\}	$, producing the scoring function 
                \begin{equation}\label{eq:step1}
                \widehat{s}
                =\mathcal{A}(\mathscr{D}_{n',m'})~.
                \end{equation}
                
                \item {\bf Univariate Rank Test.} 	Form the univariate samples
                $$	\{\widehat{s}(\bX_{1+n'}),\; \ldots,\; \widehat{s}(\bX_n)\} \text{ and } 	\{\widehat{s}(\bY_{1+m'}),\; \ldots,\; \widehat{s}(\bY_m)\}~,$$
                    the outcome of the test being finally determined by computing the binary quantity
                \begin{equation}\label{eq:step2} 
                \Phi_{\alpha}\left( \Rank(\widehat{s}(\mathbf{X}_{1+n'})),\; \ldots,\; \Rank(\widehat{s}(\mathbf{X}_n)) \right)~,
                \end{equation}
            where $\Rank(t)=\sum_{i=1+n'}^n\mathbb{I}\{\widehat{s}(\mathbf{X}_i) \leq t \}+\sum_{j=1+m'}^m\mathbb{I}\{\widehat{s}(\mathbf{Y}_j) \leq t \}$.
                \end{enumerate}
        \end{minipage} }
        \caption{Ranking-based two-sample rank test procedure.}\label{twostageproc}
        \end{figure}

        Before discussing at length the various forms the general principle described above may take, a few remarks are in order.
        \begin{remark}{\sc (Bipartite ranking algorithms)}\label{rem:BRalgo} As mentioned in sec. \ref{subsec:rationale}, the vast majority of bipartite ranking algorithms documented in the statistical learning literature solve $M$-estimation problems over specific classes $\S_0$ of scoring functions. The criterion one seeks to maximize is the $\auc$ or a (smoothed / concavified / penalized) variant,
            such as \eqref{eq:W_crit}, whose set of optimal elements coincide with a subset of $\S^*$. See for instance \cite{FISS03}, \cite{Rak04}, \cite{RCMS05}, \cite{Rud06} or \cite{lambdarank07}. Generalization results in the form of confidence upper bounds for the deficit of empirical maximizers have been established under various complexity assumptions for $\S_0$ in \cite{CLV08}, \cite{AGHHPR05}, \cite{CV07}, \cite{CDV09}, \cite{MW16} and \cite{CleLimVay21}. Stronger theoretical guarantees (\textit{i.e.} bounds for the $\sup$-norm deviation $\eqref{eq:sup_dist}$)  have also been established for alternative approaches, considering $\roc$ optimization as a continuum of cost-sensistive binary classification problems and combining $M$-estimation with nonlinear approximation methods, see \cite{CV09ieee}, \cite{CV10CA}, \cite{CVALT09} or \cite{CDV13}.
        \end{remark}
        \begin{remark}\label{rk:bipartite}{\sc ($2$-split trick)} As recalled above, nearly optimal scoring functions are generally learned by means of $M$-estimation techniques. Consequently, their dependence on the training observations may be complex and can hardly be explicit in general. For this reason,  a $2$-split trick is used in order to make the analysis of the fluctuations of the quantity \eqref{eq:step2} tractable. Hence, conditioned upon the subsamples used in the bipartite ranking step of the procedure, the functional \eqref{eq:step2} is a two-sample rank statistic. We also underline that the issue of choosing appropriately 
        the subsample dedicated to bipartite ranking and the one used for 
        the rank test is as crucial from a practical perspective as difficult, insofar as it is hard to know in advance the complexity of the ranking task.  
        \end{remark}
        We now propose several ways of implementing the methodology summarized in Fig. \ref{twostageproc}, which will be next studied theoretically in specific situations and whose performance will be empirically  investigated at length in section \ref{sec:num}.
        \medskip
        
        \noindent {\bf Ranking-based two-sample linear rank tests.} The simplest implementation consists in considering a test based on a univariate two-sample linear rank statistic \eqref{eq:R_stat}, characterized by a given score-generating function $\phi$, see Definition \ref{def:R_stat}. As recalled in sec. \ref{subsec:back_unirankroc}, such a statistic is pivotal under the homogeneity assumption $\mathcal{H}_0$ in the univariate case. Its null probability distribution can easily be  tabulated, even in the case where $(n,\; m)$ takes very large values given the computing power now at disposal. For all $n,\; m\geq 1$ and any $\alpha\in (0,1)$, one may thus determine the quantile: 
        \begin{equation}\label{eq:Rstat_quantile}
        q^{\phi}_{n,\; m}(\alpha)
        =\inf_{t\geq 0}\left\{\mathbb{P}_{\mathcal{H}_0}\left\{ \frac{1}{n}\widehat{W}_{n,m}^{\phi}-\int_0^1\phi(u)\dd u \leq t \right\}\geq 1-\alpha  \right\}~,
        \end{equation}
        as well as the critical region:
        \begin{equation}\label{eq:Rstat_quantilereg}
        \left\{ \frac{1}{n}\widehat{W}_{n,m}^{\phi} >\int_0^1\phi(u)\dd u+ q^{\phi}_{n,\; m}(\alpha) \right\}
        \end{equation}
        occuring with probability less than $\alpha$ under $\mathcal{H}_0$ in the univariate case and defining the test at level $\alpha$:
        \begin{equation}
        \Phi^{\phi}_{\alpha}(\mathscr{D}_{n,m})=\mathbb{I}\left\{ \frac{1}{n}\widehat{W}_{n,m}^{\phi} > \int_0^1\phi(u)\dd u+ q^{\phi}_{n,\; m}(\alpha)  \right\},
        \end{equation}
        based on the univariate samples $\mathscr{D}_{n,m}=\{X_1,\; \ldots,\; X_n\} \cup  \{Y_1,\; \ldots,\; Y_m\}$.
        As discussed in sec. \ref{subsec:rationale},  only a unilateral test $\Phi_{\alpha}^{\phi}$ is relevant in the multivariate case, given the reformulations \eqref{eq:2_sample_pb_alt} or \eqref{eq:2_sample_pb_alt2} of the two-sample testing problem. This contrasts 
        with the univariate situation for which no bipartite ranking step is required.
        \cite{CleLimVay21} investigated  a natural bipartite ranking approach, consisting in maximizing a statistical version of the performance criterion \eqref{eq:W_crit} based on the (multivariate) training data $\mathscr{D}_{n',m'}$ over the class $\S_0$, \textit{i.e.} in  solving the optimization problem:
        \begin{equation}\label{eq:ERM_like}
        \max_{s\in \S_0} \; \widehat{W}_{n',m'}^{\phi}(s),
        \end{equation}
        where we set for any scoring function $s(z)$:
        \begin{equation}\label{eq:W_empr_crit}
        \widehat{W}_{n',m'}^{\phi}(s)=\sum_{i=1}^{n'}\phi\left(\frac{\Rank(s(\bX_i))}{N'+1} \right),
        \end{equation}
        with $N'=n'+m'$, the quantity $\Rank(s(\bX_i))/(N'+1)$ being a natural empirical counterpart of $F_s(s(\bX_i))$ for $i=1,\; \ldots,\; n'$. \cite{CleLimVay21} studied the generalization capacity of solutions of the problem \eqref{eq:ERM_like} and (gradient ascent based) optimization strategies for approximately solving \eqref{eq:ERM_like}. Hence, by considering a solution $\widehat{s}$ of \eqref{eq:ERM_like} obtained at 
         \textit{Step 1}, the test built at \textit{Step 2} based on the scored two samples:
        \begin{equation}
        \mathscr{D}_{n'',m''}(\widehat{s})=
                \{\widehat{s}(\bX_{1+n'}),\; \ldots,\; \widehat{s}(\bX_n)\} 
                \cup	\{\widehat{s}(\bY_{1+m'}),\; \ldots,\; \widehat{s}(\bY_m)\}~,
        \end{equation}
        with $n'' = n-n'$ and $m''=m-m'$, writes
        \begin{equation}\label{eq:rb2sample_test}
        \Phi^{\phi}_{\alpha}(\mathscr{D}_{n'',m''}\left(\widehat{s})\right)
        =\mathbb{I}\left\{\frac{1}{n''}\widehat{W}^{\phi}_{n'',m''}(\widehat{s})>\int_0^1\phi(u)\dd u+q_{n'',m''}^{\phi}(\alpha)\right\}.
        \end{equation}
        Under specific assumptions,  in particular related to the class $\S_0$ and the score-generating function $\phi$, the nonasymptotic properties of the test \eqref{eq:rb2sample_test} are investigated in sec. \ref{subsec:theory}. 
        
        \begin{remark}{\sc (Combining multiple ranking-based two-sample linear rank tests)}
        As highlighted in \cite{CleLimVay21}, depending on the score-generating function $\phi$ chosen, the quantity $W^*_{\phi}$ summarizes $\roc^*$ in a certain fashion. As one hardly knows in advance which $\phi$ may capture best the way $\roc^*$ possibly deviates from the diagonal in practice (through the scalar quantity $W^*_{\phi}-\int_0^1\phi(u)du\geq 0$), a natural strategy could consist in performing simultaneously  several ranking-based two-sample linear rank tests, implementing popular principles in multiple hypothesis testing and ensemble learning. For instance, $K\geq 1$ score generating functions $\phi_1,\; \ldots,\; \phi_K$ can be considered, together with levels $\alpha_1,\; \ldots,\; \alpha_K$ to form the ensemble of tests $\{\Phi^{\phi_k}_{\alpha_k}: k=1,\; \ldots,\; K\}$ and the combination
        \begin{equation}\label{eq:comb}
        \sup_{1\leq k\leq K} \Phi^{\phi_k}_{\alpha_k} \left(\mathscr{D}_{n'',m''}( \widehat{s})\right).
        \end{equation}
        Although one may guarantee that the type I error of \eqref{eq:comb} is less than $\alpha$ (by choosing the $\alpha_k$'s so that $\sum_{k=1}^K\alpha_k=\alpha$), one faces significant difficulties when investigating its properties, due to the dependence of the tests combined and of the different $\phi_k$. The study of such approaches are thus left for further research.
        \end{remark}

        \noindent {\bf Ranking-based tests in the $\roc$ space.}  As underlined in sec. \ref{subsec:back_unirankroc}, the empirical $\roc$ curve is itself a two-sample rank statistic (and consequently pivotal) under $\mathcal{H}_0$. Hence, the test involved at \textit{Step 2} can be based on a confidence region $\mathcal{R}_{n'',m''}(\alpha)\subset \mathcal{C}_{n'',m''}$ at level $1-\alpha$ for the empirical $\roc$ curve based on univariate samples of sizes $n''$ and $m''$ under $\mathcal{H}_0$. Using the scoring function $\widehat{s}$ produced at \textit{Step 1}, one plots the empirical $\roc$ curve related to the univariate distributions:
        \begin{equation}
        \widehat{H}_{\widehat{s},n''}=\frac{1}{n''}\sum_{i=1+n'}^n\delta_{\widehat{s}(\bX_{i})}   \text{ and } \widehat{G}_{\widehat{s},m''}=\frac{1}{m''}\sum_{j=1+m'}^n\delta_{\widehat{s}(\bY_{j})}
        \end{equation}
        and the critical region then writes
        \begin{equation}\label{eq:crit_region}
        \left\{ \roc_{\widehat{H}_{\widehat{s},m''},\; \widehat{G}_{\widehat{s},n''}} \notin \mathcal{R}_{n'',m''}(\alpha)  \right\}.
        \end{equation}
        
        As pointed out in sec. \ref{subsec:back_unirankroc}, given a (pseudo-)metric in the $\roc$ space (\textit{e.g.} the $\sup$ norm), the confidence region $\mathcal{R}_{n'',m''}(\alpha)$ could naturally correspond to the set of piecewise linear curves in $\mathcal{C}_{n'',m''}$ at a distance smaller than a specific threshold $t_{\alpha}$ from the diagonal (and possibly above the diagonal, just like $\roc^*$). In the case of the $L_1$-distance, this implementation coincides with the previous one when choosing $\phi(u)=u$. 
        
        \subsection{Theoretical Guarantees - Nonasymptotic Error Bounds}\label{subsec:theory}
        The properties of the ranking-based two-sample linear rank tests  described in sec. \ref{subsec:method} are now analyzed from a nonasymptotic perspective.  Let $\alpha\in (0,1)$. The test \eqref{eq:rb2sample_test} is of level $\alpha$ by construction as proved by  Theorem \ref{thm:typeI}  below.
        
        \begin{theorem}\label{thm:typeI}{\sc(Type-I error bound)} Let $\phi(u)$ be a score-generating function and $n,\; m\geq 2$. Fix $\alpha\in (0,1)$.
        Under the null hypothesis $\mathcal{H}_0$, the type-I error of the test \eqref{eq:rb2sample_test} is less than $\alpha$
        \begin{equation}
        \mathbb{P}_{\mathcal{H}_0}\left\{ \Phi^{\phi}_{\alpha}\left( \mathscr{D}_{n'',m''}(\widehat{s}) \right) = 1  \right\}  \leq \alpha~,
        \end{equation}
        for all $1\leq n''< n$ and $1\leq m'' < m$.
        \end{theorem}
        \noindent The proof directly follows from a conditioning argument detailed in Appendix section \ref{app:proofs}. The distribution of \eqref{eq:R_stat} can be easily tabulated and the quantile \eqref{eq:Rstat_quantile}  numerically computed for any $n,\; m\geq 1$, see sec. \ref{subsec:method}. Consider the following assumption related to the smoothness of the score-generating function $\phi$.
        
        \begin{hyp}\label{hyp:phic2}
            The score-generating function $\phi : [0,1] \mapsto \RR$, is nondecreasing and twice continuously differentiable.
        \end{hyp}
        \noindent Under  Assumption \ref{hyp:phic2}, the following result provides an upperbound for  \eqref{eq:Rstat_quantile} that decays to $0$, as simultaneously $n,\; m$ tend to infinity (so that $n/(n+m)\rightarrow p$ at the rate $1/(n+m)$ for $p\in (0,1)$).
        
        \begin{proposition}\label{prop:rate_quantile} Suppose Assumption \ref{hyp:phic2} fulfilled. Let $p\in (0,1)$ and $N\geq 1/p$. Set $n = \lfloor pN \rfloor$ and $m = \lceil (1-p)N \rceil = N - n$. Then, for any $\alpha\in (0,1)$, we have:
        \begin{equation}
         q^{\phi}_{n,m}(\alpha) \leq \sqrt{\frac{\log(18/\alpha)}{CN}}~,
        \end{equation}
        where 
        $C=8^{-1}\min\left(p/\lVert \phi \rVert_{\infty}^2,  (p \lVert \phi' \rVert_{\infty}^2)^{-1}, ((1-p) \lVert \phi' \rVert_{\infty}^2)^{-1} \right) $.
        \end{proposition}
        \noindent Refer to  Appendix \ref{app:proofprop8} for the proof. It straightforwardly results from the tail probability bound for two-sample linear rank statistics established in  Appendix \ref{app:proofs}, Theorem \ref{thm:tail}, the bound being of the expected order given the CLT satisfied by \eqref{eq:R_stat}, refer to \textit{e.g.} Theorem 13.25 in \cite{vdV98}.\\
        
         In addition, the type-II error can also be controlled using nonasymptotic results proved in \cite{CleLimVay21}. They provided probability inequalities for the maximal deviations between the empirical criterion \eqref{eq:W_empr_crit} and the theoretical one \eqref{eq:W_crit},  over a class $\S_0$ of scoring functions of controlled complexity, and upper confidence bounds for the deficit of $W_{\phi}$-ranking performance of solutions of problem \eqref{eq:ERM_like}. The latter results rely on linearization techniques applied to the statistic \eqref{eq:W_empr_crit} and concentration bounds for two-sample $U$-processes. As will be shown below and given $\widehat{s}\in \S_0$, these results permit to establish tail bounds for the quantity
        
        \begin{multline}\label{eq:dec1}
        \frac{1}{n''}\widehat{W}^{\phi}_{n'',m''}(\widehat{s})-\int_{0}^1\phi(u)\dd u
        =\left\{ \frac{1}{n''}\widehat{W}^{\phi}_{n'',m''}(\widehat{s})-W_{\phi}(\widehat{s})\right\}
        + \left\{ W_{\phi}(\widehat{s})  - W^*_{\phi}\right\}\\  +\left\{ W^*_{\phi} - \int_{0}^1\phi(u)\dd u \right\}.
        \end{multline}
        
        \noindent The (non-negative) third term on the right hand side of Eq. \eqref{eq:dec1} quantifies the deviation from the homogeneity hypothesis $\mathcal{H}_0$, while the second one corresponds to the error inherent in the bipartite ranking step (\textit{Step 1} in Fig. \ref{twostageproc}). The following additional assumptions are required to apply those results.
        
        \begin{hyp}\label{hyp:sabscont}
        Let $M>0$.	For all $s\in \S_0$, the random variables $s(\bX)$ and $s(\bY)$ are continuous, with density functions that are twice differentiable and have Sobolev $\mathcal{W}^{2,\infty}$-norms\footnote{Recall that the Sobolev space $\mathcal{W}^{2,\infty}$ is the space of all Borelian functions $h:\mathbb{R}\rightarrow \mathbb{R}$ such that $h$ and its first and second order weak derivatives $h'$ and $h''$ are bounded almost-everywhere. Denoting by $\vert\vert.\vert\vert_{\infty}$ the norm of the Lebesgue space $L_{\infty}$ of Borelian and essentially bounded functions, $\mathcal{W}^{2,\infty}$ is a Banach space when equipped with the norm $\vert\vert h\vert\vert_{2,\infty}=\max\{\vert\vert h\vert\vert_{\infty},\; \vert\vert h'\vert\vert_{\infty},\; \vert\vert h''\vert\vert_{\infty}   \}$.} bounded by $M<+\infty$.
        \end{hyp}

        \begin{hyp}\label{hyp:VC}
            The class of scoring functions $\S_0$ is a {\sc VC} class of finite {\sc VC} dimension $\V<+\infty$.
        \end{hyp}
        
        \noindent We refer to section $2.6.2$ in \cite{vdVWell96} for the definition of  \textit{Vapnik–Chervonenkis} (\vc) classes of functions.
        \cite{CleLimVay21} precisely proved, under the assumptions above and when $n'/(n'+m')\sim n''/(n''+m'')\sim p$, that the deficit of $W_{\phi}$-ranking performance of solutions of empirical maximizers $\widehat{s}$, \textit{i.e.}, the second term on the right hand side of \eqref{eq:dec1}, is of order $O_{\mathbb{P}}(1/\sqrt{n'\wedge m'})$ (when neglecting the possible bias model, arising from the fact that $\S_0\cap \S^*$ might be empty, see Corollary 7 therein), and that the first term is of order $O_{\mathbb{P}}(1/\sqrt{n''\wedge m''})$ (see the argument of Theorem 5 therein).\\
        
         Considering the quantity $W^*_{\phi}-\int_{0}^1\phi(u)\dd u$ to describe the departure from the null assumption $\mathcal{H}_0$ (see Proposition \ref{prop:rationale}) and the bias model $W^*_{\phi}-\sup_{s\in \S_0}W_{\phi(s)}$ inherent in the bipartite ranking step (when formulated as empirical $W_{\phi}$-ranking performance maximization), we introduce the two (nonparametric) classes of pairs of probability distributions on $\mathcal{Z}$.
        
        \begin{definition}\label{def:depart} Let $\varepsilon>0$ and $\phi(u)$ be a score-generating function.
            We denote by $\mathcal{H}_1(\varepsilon)$ the set of alternative hypotheses corresponding to all pairs $(H,G)$ of probability distributions on $\mathcal{Z}$ such that
        \begin{equation*}
        W^*_{\phi}-\int_{0}^1\phi(u)\dd u \geq \varepsilon~,
        \end{equation*}
        where we recall $W_{\phi}^*= W_{\phi}(s^*) =  W^*_{\phi}(H,G) $ for any $s^*\in \S^*$.
        \end{definition}
        
        \begin{definition}\label{def:bias}
        Let $\delta>0$, $\phi(u)$ be a score-generating function and $\S_0$ be a class of scoring functions. We denote by $\mathcal{B}(\delta)$ the set of  all pairs $(H,G)$ of probability distributions on $\mathcal{Z}$ such that
        \begin{equation*}
        W^*_{\phi}- \sup_{s\in \S_0}W_{\phi}(s) \leq \delta~.
        \end{equation*}
        \end{definition}
        The theorem below provides a rate bound for the type-II error of the ranking-based rank test \eqref{eq:rb2sample_test} of size $\alpha$, depending on the sizes $N'$ and $N''=N-N'$ of the two pooled samples involved in the procedure, $\ie$, that are $\resp$ used for bipartite ranking and for performing the rank test based on the scoring function learned.
        
        \begin{theorem}\label{thm:typeII} {\sc (Type-II error bound)} Let $\phi(u)$ be a score-generating function and $\varepsilon>\delta>0$. Fix $\alpha\in (0,1)$. Suppose that Assumptions \ref{hyp:phic2}-\ref{hyp:VC} are fulfilled. Let $p \in (0,1)$ such that $N'\wedge N''\geq 1/p$. Set $n'= \lfloor p N' \rfloor$ and $m' = \lceil (1-p)N' \rceil = N' - n'$, as well as  $n''= \lfloor pN'' \rfloor$ and $m'' = \lceil (1-p)N'' \rceil = N'' - n''$.j
        Then, there exist constants $C_1$ and $C_2 \geq 24$, depending on $(\phi, \; \V)$, such that the  type-II error of the test \eqref{eq:rb2sample_test} is uniformly bounded as follows:
        \begin{multline}\label{eq:typeII}
        \sup_{(H,G)\in \mathcal{H}_1(\varepsilon)\cap \mathcal{B}(\delta)}\mathbb{P}_{H,G}\left\{ \Phi^{\phi}_{\alpha}(\mathscr{D}_{n'',m''}\left(\widehat{s})\right) =0  \right\} 
         \leq  18 \exp\left\{-\frac{CN''}{16}\left( \varepsilon - \delta)\right)^2\right\} \\
         + C_2 
          \exp\left\{-\frac{N'}{8C_2}p(p\wedge  (1-p))( \varepsilon - \delta)\log\left(1+\frac{\varepsilon-\delta}{32C_1(p\wedge  (1-p))}\right) \right\}~,
        \end{multline}
        as soon as $N''\geq 4\log(18/\alpha)/(C(\varepsilon-\delta)^2)$ and $N'\geq 16C_1^2/(p(\varepsilon-\delta)^2)$, the constant $C$ being that involved in Proposition \ref{prop:rate_quantile}, the $C_j$'s those involved in Theorem \ref{thm:concentration}.
        \end{theorem}
        \noindent Refer to  Appendix \ref{app:prooftypeII} for the detailed proof.  In view of Definitions \ref{def:depart} and \ref{def:bias}, by requiring $\varepsilon>\delta$, we consider here pairs $(H,G)$ for which the class $\mathcal{S}_0$ induces a model bias in the bipartite ranking problem that is sufficiently small, namely smaller than the minimum departure $\varepsilon$ from the null assumption. The bound stated above reveals that, when $\varepsilon-\delta>0$ is held fixed,  the type-II error uniformly vanishes exponentially fast as $N'$ and $N''$ simultaneously increase to infinity: the second term on the right hand side of the bound \eqref{eq:typeII} is inherent in the learning stage, while the first term corresponds to a bound for the type-II error of a univariate rank test. Hence, for $N'$ large enough, the type-II error of the ranking-based rank test is of the same order of magnitude as that of a rank test based on the univariate samples $\{\Psi(\bX_{1+n'}),\; \ldots,\; \Psi(\bX_n)\}$ and $\{\Psi(\bY_{1+m'}),\; \ldots,\; \Psi(\bY_m)\}\}$.  The two bounds exhibit a different behavior as the departure level $\varepsilon$ from the null assumption tends to $0$ or to $+\infty$: when the pooled sample sizes $N'$ and $N''$ are held fixed, the bound related to the type-II error of the univariate rank test becomes negligible compared to the learning bound as $\varepsilon$ increases to infinity, while it deteriorates faster than it when  $\varepsilon$ decays to $0$.
        Observe also that in the result stated above, the fraction $p$ of positive instances within the pooled sample involved in the bipartite ranking step  (\textit{Step 1}) is the same as that within the pooled sample used in the rank test (\textit{Step 2}), the same rank statistic $\widehat{W}^{\phi}$ being used for both steps in the present analysis.  However, if the choice of $N''$ can be guided by the value of the size $\alpha$ considered in view of the discrete pivotal conditional distribution of the test statistic used, the bound established in Theorem \ref{thm:typeII}  does not permit to tune in practice the amount of the $N$ training observations that should be used for each \textit{Step 1} and \textit{Step 2}. 
        As will be discussed in section \ref{sec:num}, simply dividing the samples into two parts, with a larger part of the pooled dataset kept for the learning stage (\textit{e.g.} $N'\sim 4N/5$)  may permit to get satisfactory results.

        \begin{remark}{\sc (Relation to minimax testing)} The proposed definition  for the alternative assumption $\cH_1$ in the sense of the $W_{\phi}$-criterion is novel compared to the nonparametric minimax testing literature. The latter formulates  $\cH_1$  by measuring the dissimilarity of the underlying distributions using a distance $\wrt$ a particular metric, see \textit{e.g.}, \cite{LWCarpSri22}  for local minimax separation rate  defined  by  $L_1$-norm for discrete distributions, \cite{Carp18} for $L_2$-norm in sparse linear regression. In fact, Definition \ref{def:depart} offers a similar interpretation when consider the $W_{\phi}$-criterion in the $\roc$ space, particularly highlighted when choosing $\phi(u)=u$ by Eq. \eqref{eq:aucdist} and Remark $5$ in \cite{CleLimVay21}. Thus, one could formulate minimax properties of the type-II error of the present statistic, see discussion in \cite{LimniosPhD}, Chap. 6.4.2 therein. 
        \end{remark}

         We finally point out that  similar results to Theorem \ref{thm:typeII} can be established for the ranking-based test corresponding to the critical region \eqref{eq:crit_region} in the $\roc$ space when the latter is based on the $\sup$-norm,  exactly in the same way except that  bounds (in $\sup$ norm) established in \cite{CV10CA} or in \cite{CV09ieee} must be used instead of those in \cite{CleLimVay21}.

        \section{Illustrative Numerical Experiments}\label{sec:num}
        In this section, we illustrate the methodology previously described, and analyzed in the case of $W_{\phi}$-ranking performance optimization, by displaying the results of various numerical experiments, based on synthetic datasets. Beyond the empirical evidence of its capacity to detect (small) departures from the homogeneity assumption successfully, compared to certain popular two-sample tests in the multivariate setup, they aim at showing the impact of the various ingredients involved, the score generating function $\phi$ chosen and the bipartite ranking algorithm used especially.
        We first detail the variants of the ranking-based approach to the two-sample problem examined in the experiments. Next the two-sample problems considered are specified and the empirical results are finally summarized and discussed. Additional experimental results are postponed to the Appendix section  \ref{sec:app_expe}. All the Python codes used to carry out these experiments are available online at \url{https://github.com/MyrtoLimnios/twosampleranktest} for reproducibility purpose. 
        
        \subsection{Practical Implementations - Hyperparameters}\label{subsec:numimpl}
        
        The two-stage testing procedure introduced in subsection \ref{subsec:method} can be implemented in various ways, depending on the bipartite ranking algorithm $\A$ used and the score-generating function chosen.
        \medskip
        
        \noindent{\bf Bipartite ranking algorithms (\textit{Step 1}).} As recalled in subsection \ref{subsec:rationale} (see also Remark \ref{rk:bipartite}), most techniques documented in the machine learning literature  formulate bipartite ranking as a \textit{pairwise classification} problem, see \cite{CLV08} and  \cite{MW16}. Hence, popular classification algorithms (\textit{e.g.} SVM, neural networks, boosting) can be applied to pairs $(Z,Z')$ to which label $-1$ is assigned when $Z$ and $Z'$ are drawn from the same multivariate distribution ($H$ or $G$) and label $+1$ otherwise. In particular, for the present experiments we implemented the linear version of  RankSVM with $L_2$ loss (\texttt{rSVM2}, see \cite{JoaRSVM02}), RankNN (\texttt{rNN}, see  \cite{ranknet05}) and 
         RankBoost (\texttt{rBoost}, see \cite{FISS03}). Beyond the pairwise approach, recursive partitioning methods based on oriented binary trees have been proposed to optimize an adaptively discretized version of the $\roc$ curve (with remarkable generalization guarantees in $\sup$ norm), see \cite{CV09ieee} and \cite{CDV09} (see also \cite{CV10CA} and \cite{CVALT09}). Here we have implemented the stabler ensemble learning version referred to as Ranking Forest (\texttt{rForest}, see \cite{CDV13}), whose practical performance is assessed in \cite{CDV13bis} in particular. In addition, as explained in \cite{CleLimVay21} (see section 4 therein), a regularized version of the empirical $W_{\phi}$-ranking performance criterion can be obtained by a kernel smoothing procedure and maximized through a (stochastic) gradient ascent algorithm,  akin to that theoretically analyzed in subsection \ref{subsec:theory}.
         \medskip
        
        \noindent{\bf Univariate two-sample rank tests (\textit{Step 2}).} Let $\hat{s}$ be the outcome of the first step. In order to perform \textit{Step 2}, we considered the test statistic $\widehat{W}_{n,m}^{\phi}(\hat{s})$ with the following score generating functions: $\phi_{MWW}(u) = u$, which corresponds to the classic Mann-Withney-Wilcoxon statistic (MWW, \cite{Wil45}), and $\phi_{RTB}(u) = u\mathbb{I}\{ u\geq u_0\}$ for $u_0\in (0,1)$ (RTB, \cite{CV07}) so as to focus on top ranks and the behavior of the $\roc$ curve near the origin.
        
        For a given level $\alpha$, the type-I error is controlled by taking as critical threshold the quantile $q_{n'',m''}^{\phi}(\alpha)$, while the type-II error is estimated by Monte Carlo for the pairs of instrumental distributions described below. As discussed in subsection \ref{subsec:back_unirankroc}, the rank statistics are pivotal and the quantiles $q_{n'',m''}^{\phi}(\alpha)$ are tabulated for classic choices of $\phi$ (easily accessible by means of  the \texttt{SciPy} open-access library available in \texttt{Python} for instance) or can be easily calculated, even for very large values of $n''$, $m''$,  using modern computing frameworks.
        
        We first consider both probabilities when the \textit{dissimilarity/discrepancy} parameter $\varepsilon>0$ varies with fixed design, then we fix the dissimilarity parameter and let the dimensionality $d$  increase.
        \medskip
        
        \noindent{\bf Two-split trick.} Since the most challenging step is undoubtedly the first one, consisting in learning to rank observations nearly in the same order as that induced by the likelihood ratio $\Psi(z)$, a much larger fraction of the data is allocated to bipartite ranking. Namely, we take $N'/N=4/5$ (so that $N''=N/5$).
        
        \medskip
        \noindent{\bf Benchmark tests.}
        The performance of the ranking-based methodology proposed is compared to that of four state-of-the-art multivariate and nonparametric two-sample tests:  namely the unbiased (quadratic) Maximum Mean Discrepancy (\texttt{MMD}) test with Gaussian kernels (see  \cite{GBRSS07,GBRSS12}), the graph-based Wald-Wolfowitz runs test (\texttt{FR}) generalized to the multivariate setting as proposed in  \cite{FriRaf79}, the metric-based Energy test (\texttt{Energy}) (see \cite{SzRi13}) and the depth-based procedure  (\texttt{Tukey}) documented in \cite{LiuS93} to extend univariate two-sample rank tests with the Tukey statistical depth (\cite{Tukey75}), see subsection \ref{subsec:back_unirankroc} for further details.  Notice that both \texttt{MMD} and \texttt{Energy} are not exactly distribution-free tests under $\mathcal{H}_0$, therefore they are calibrated by means of a permutation technique. We also point out that \texttt{Tukey} is implemented by using the same tabulations of rank statistics as those used to apply the 
          ranking-based methodology.
        
        \medskip
        
        \noindent{\bf Evaluation criteria.}  The frequencies of type-I error ($\varepsilon = 0$) and of type-II error ($\varepsilon > 0$) of each testing procedure at all levels $\alpha\in (0,1)$ have been computed over $B\geq 1$ Monte-Carlo replications for each experiment and the distributions of $p$-values obtained are also reported. The impact of an increase of the dimension $d$ on the power, for fixed $\varepsilon$ and sample size $N$, is investigated as well.
        
        \medskip
        
        \noindent{\bf Experimental parameters.} For all the experiments, the pooled sample is of size $N = 2000 $ and balanced, \textit{i.e.} $n= m=N/2$ (corresponding to $p=1/2$). The number of Monte Carlo replications is $B = 100$. The parameter $u_0$ for the RTB score-generating function varies  in the set $\{0.7,0.8,0.9\}$. For the benchmark tests, the null distribution is estimated over $B_{perm} = 1000$ permutations and the related hyperparameter is optimized over the range $ \{1e-3, 1e-2, 1e-1, 1, 5, 10, 15, 20, 25, 30, 1e2, 1e3\}$. The figures also depict the pointwise confidence intervals at level $95\%$. 
        
        \subsection{Synthetic Datasets}
        We illustrate the performance of the family of rank-based tests proposed through the following two-sample problems. Various location and scale Gaussian models are considered in order to compare it to that of the (optimal) likelihood ratio statistic, see Example \ref{exgauss}. We also test the homogeneity samples drawn from mixture models and from less classic (heavy-tailed) statistical distributions. 
        See the Appendix section \ref{sec:app_expe} for additional information on the  parameters chosen for the distributions and the plots of the  true $\roc$ curves (including $\roc^*$). For $d\geq 1$, denote by $0_d\in \RR^d$ the null vector,  by $\mathbf{1}_d\in \RR^d$ the unit  vector ($\ie$ with all coordinates equal to $1$), by $ \I_d\in\RR^{d\times d}$ the identity matrix and by $S_d^+(\RR)$ the cone of positive semi-definite $d\times d$ matrices with real entries.
        \medskip
        
        \noindent{\bf Location  Gaussian models.} The two samples $\bX \sim \mathcal{N}_d(\mu_X, \Sigma)$ and $\bY \sim \mathcal{N}_d(\mu_Y, \Sigma)$ are drawn independently, with $\Sigma \in S_d^+(\RR)$, $\varepsilon \in \{0.0, 0.02, 0.05, 0.08, 0.1\}$, as follows.
        
        \begin{enumerate}
            \item[(L1)] We set $\mu_Y = 0_d$ and  $\mu_X  = (\varepsilon / \sqrt{d})\times \mathbf{1}_d$. Two models for the covariance matrix are considered: the first coordinate  is negatively correlated with all the others and  the other $d-1$ coordinates are : mutually independent (L1$-$); positively correlated (L1$+$).
        \end{enumerate}
        
        For such models, the class $\S^*$  can be explicited, as detailed in Example \ref{exgauss}.
        \medskip
        
        \noindent{\bf Scale  Gaussian models.} The two samples $\bX \sim \mathcal{N}_d(0_d, \Sigma_X)$ and $\bY \sim \mathcal{N}_d( 0_d, \Sigma_Y)$  are drawn independently with $\Sigma_X, \; \Sigma_Y$ in $S_d^+(\RR)$, defined as follows.

        \begin{enumerate}
            \item[(S1)]	\textit{Decreasing correlation.} Set $\Sigma_{X,i,j}  = \alpha^{ \vert i-j\vert} $ and $\Sigma_{Y,i,j}  = \beta^{ \vert i-j\vert} $ for $i,\; j \leq d$, with $d = 20 $, $\beta= 0.2$ and $\alpha = \beta + \varepsilon$, $\varepsilon \in \{0, 0.1, 0.2, 0.3\}$.
            \item[(S2)] \textit{Equi-correlated samples.} Set $\Sigma_X = (1-\alpha) \I_d + \alpha \mathbf{1}_d\mathbf{1}_d^T$ and $\Sigma_Y = (1-\beta) \I_d + \beta \mathbf{1}_d\mathbf{1}_d^T$, with $d \in \{20, \; 30, \; 60, \; 100\}$, $\beta= 0.3$ and $\alpha = \beta + \varepsilon$, $\varepsilon \in \{0, 0.05, 0.1, 0.15\}$.	
        \end{enumerate}
        
        As for the Gaussian  location model, an optimal scoring function can be easily explicited in this case, the quadratic function  $s_{\theta^*} (\cdot) = \langle
        \cdot, \theta^{*} \cdot \rangle$, where $\theta^* = \Sigma_X^{-1} - \Sigma_Y^{-1}$ for instance.
        \medskip
        
        \noindent{\bf Non Gaussian models.} We also consider generative models with heavier tails to build other two-sample problems.

        \begin{enumerate}
            \item[(T1)] \textit{Cauchy distribution.} We generate $\bX = (X_1, X_2, X_3)$ and $\bY= (Y_1, Y_2, Y_3)$ independently, such that $X_1, X_2$ $  \overset{i.i.d}{\sim} \text{Cauchy}(\varepsilon, 1)$ and  $X_3,\; Y_1,\; Y_2,\; Y_3 \overset{i.i.d}{\sim} \text{Cauchy}(0, 1)$ with $\varepsilon \in \{0, 0.05, 0.15, 0.20\}$, $d = 3$.
            \item[(T2)]  \textit{Lognormal location.} We generate $\bX$ and $\bY$, so that $\log (\bX)$ and $\log (\bY)$ are drawn from the (L1) model with $\varepsilon \in \{0, 0.1, 0.2, 0.3\}$, $d=4$. 
            \item[(T3)]  \textit{Lognormal scale.} We generate $\bX$ and $\bY$, so that $\log \bX$ and $\log \bY$ are drawn from the (S1) model with $\varepsilon \in \{0, 0.1, 0.2, 0.3\}$, $d=20$. 
        \end{enumerate}

        The model (L1) is used in \cite{CleLimVay21} in the bipartite ranking context, the models (S1), (S2) and  (T1) are considered in \cite{DebSen19}, whereas the models (T2) and  (T3) are widely used in the univariate setting.

        \subsection{Results and Discussion}
        We now discuss  the numerical results obtained by the approach for the two-sample problems we have proposed and previously described. We compare its performance, depending on the bipartite ranking algorithm chosen for \textit{Step 1} and the function $\phi$, to some state-of-the-art (SoA) nonparametric tests  on the basis of the type-I/II error counts. 
        Focus is naturally on the  ability of each method to reject $\cH_0$ for small departures from it (the type of departure considered varying as well across the experiments), $\ie$ when $\varepsilon\to 0$, while controlling the type-I error.
        Precisely, we discuss for all methods : 1) their ability to control the type-I error for the range of  levels $\alpha \in (0, 1)$ (see the graphs), 2) the distribution of their  $p$-values depending on $\varepsilon$ (see the boxplots), and 3) their ability to reject the null hypothesis for the range of  levels $\alpha \in (0,1)$ as the dimension of the feature space $d$ increases  (see the graphs).  
        The various tests are tuned so as to control the type-I error while maximizing the power at fixed level $\alpha = 0.05$. They are compared for the most difficult two-sample problems considered, \textit{i.e.} the smallest value of $\varepsilon$  (see the tables).  Here, the results are displayed for the score generating function $\phi(u) = u$, which corresponds to the MWW  test statistic. In the Appendix section, some results are also given for other score-generating functions in order to show how the choice of $\phi$ may possibly impact the performance of the two-sample test (\textit{Step 2}), in particular that of the RTB version. 
          In addition to the figures and tables, we successively review the results obtained for each of the two-sample problems considered. \\  
        
        As a first go, we analyze the results for the Gaussian location models.  
        For such two-sample problems, as evidenced by Table \ref{tab:powerL}, ranking-based tests clearly manifest a greater capacity to reject the null hypothesis for small values of $\varepsilon$  than SoA tests,  while achieving a better control of the type-I error, see Figures  \ref{fig:boxplotH1_L1-}  (a)  and  \ref{fig:boxplotH1_L1+}   (a), in particular \texttt{rForest}. 
        The distributions of the $p$-values exhibit lower variance and higher power for very small $\varepsilon=0.02$ (L1-)  for ranking-based methods compared to SoA tests, see Fig.  \ref{fig:boxplotH1_L1-} (b-d), as well as  Fig. \ref{fig:boxplotH1_L1+} (b-d) for (L1+), gathered in Table \ref{tab:powerL}. \\ 
        
        For the Gaussian scale models, we investigate the performance of the methods in higher dimensions. 
        The results clearly reveal that all ranking-based tests control better the type-I error than SoA procedures, see Fig. \ref{fig:boxplotH1_S2} (a).  
        In lower dimensions, the power of SoA tests is competitive, see Table \ref{tab:powerST}, except for \texttt{Tukey} that exhibits a high $p$-value variance and a low power.  In particular, both \texttt{rForest} and \texttt{rNN} have the highest rejection rates for the smaller $\varepsilon$ (see columns (S2) $\varepsilon=0.1$ and (S3) $\varepsilon=0.05$) compared to the very low rates of SoA methods.
        Additionally, when the dimension $d$ of the feature space increases, the empirical power of ranking-based test \texttt{rNN} is always greater than that of any other method, whatever the test level $\alpha\in (0,1)$, see Fig. \ref{fig:power_size_S3} for (S2) and $d>30$. \\

        For the third class of models, we analyze the distributions of the $p$-values in Fig. \ref{fig:boxplotH1_T1} for (T1),  Fig.  \ref{fig:boxplotH1_T2}  for (T2) with (L1+), and  Fig.  \ref{fig:boxplotH1_T3} for (T3) with (S1). \texttt{rBoost} and  \texttt{MMD} perform similarly for (T1),  $\ie$ multivariate Cauchy distribution. 
        \texttt{rForest} and \texttt{MMD} have comparable results for both (T2) and (T3). 
        Overall the three models and under the alternative, \texttt{rForest} shows higher empirical rejection rates for the smallest $\varepsilon$.
        Lastly, the rejection rate under the null is better controled by ranking-based algorithms see Fig. \ref{fig:boxplotH1_T2} (a) and Fig. \ref{fig:boxplotH1_T3}  (a). \\

        Finally, we tested the impact of the increase of  $d$  on the rejection rate under alternatives for all methods $\wrt$ the range of levels $\alpha \in (0,1)$ for the (S2) two-sample problem. Figure \ref{fig:power_size_S3} clearly shows that for the smallest $\varepsilon = 0.05$, the empirical power of the ranking-based method \texttt{rNN} is always above the SoA procedures, with a notable difference for the highest dimension considered, namely $d=100$. When  $\varepsilon$ increases, then \texttt{rNN} and \texttt{MMD} have similar performance. Table \ref{tab:powersizeS3} gathers the estimated power for $\alpha=0.05$ and illustrates the non-decreasing power of the ranking-based method with the dimension, where SoA methods overall have constant/decreasing power $\wrt$ the dimension, see \texttt{Energy} $\varepsilon = 0.1$.\\
        
        To conclude, we empirically illustrated the competitiveness of ranking-based rank tests, especially for the location and pathological statistical problems, as it overall shows  a clear control of the empirical type-I error for a large range of levels $\alpha$, while resulting to similar or higher rejection rates under  alternatives for very small deviations $\varepsilon\to 0$. From these experiments, the high performer ranking-based algorithm for \textit{Step 1} shows to be \texttt{rForest}, while the lower is \texttt{rBoost}. The comparative SoA tests overall  control the type-I error for the majority of  models but at the price of algorithmic complexity due to the high number of permutations. Additionally, their empirical power for very small $\varepsilon$ is lower than this of the ranking-based methods. 
        The depth-based test (\texttt{Tukey}) particularly under performs.  Lastly, the distributions of the $p$-values show large variance for decreasing $\varepsilon$ and are especially larger for the SoA methods, independently of the underlying probabilistic model.

        \begin{figure}[ht!]
            \begin{tabular}{cccc}
                
                \parbox{0.23\textwidth}{	
                        \includegraphics[width=.2\textwidth]{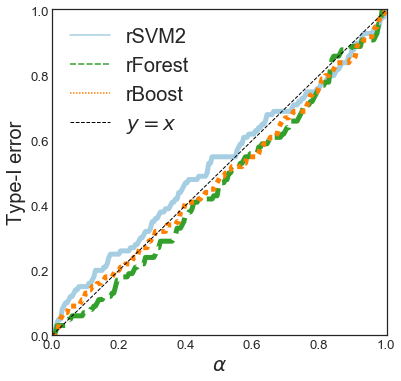}
                    \includegraphics[width=.2\textwidth]{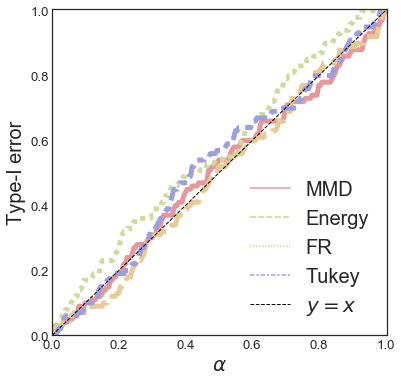}\\
                    {\small  a.  $\cH_0$,  $ \varepsilon=0.0$}
                }
            \parbox{0.23\textwidth}{	
            \includegraphics[width=.24\textwidth]{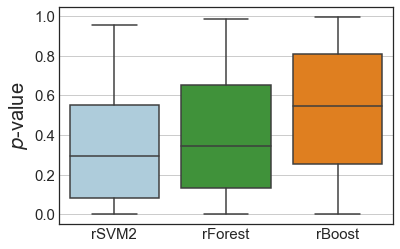}
            \includegraphics[width=.24\textwidth]{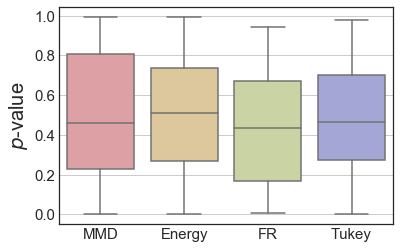}
            {\small b.  $\cH_1$,  $ \varepsilon=0.02$}
        }
                \parbox{0.23\textwidth}{	
                    \includegraphics[width=.23\textwidth]{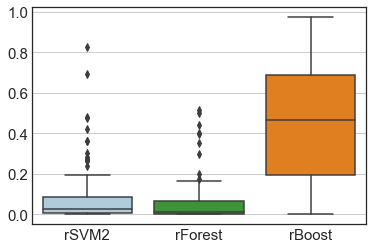}
                    \includegraphics[width=.23\textwidth]{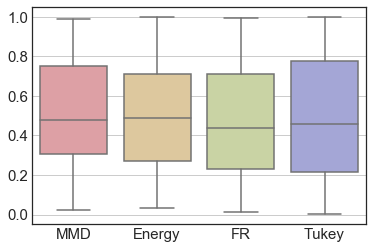}
                    {\small c.  $\cH_1$,  $ \varepsilon=0.05$}
                }	
                \parbox{0.23\textwidth}{	
                \includegraphics[width=.23\textwidth]{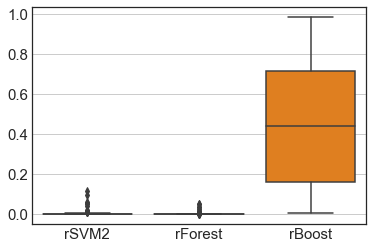}
                \includegraphics[width=.23\textwidth]{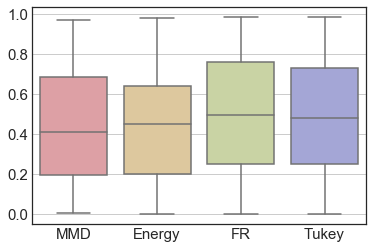}
                {\small d. $\cH_1$, $ \varepsilon=0.08$}
            }	
            \end{tabular}
            \caption{ Rejection rates of the homogeneity assumption under $\cH_0$ against significance level $\alpha \in (0, 1)$ in Fig. (a). Boxplots of  $p$-values under $\cH_1$ in Fig. (b-d)  for (L1-) with $\phi(u)=u$, $n=m=1000$, $d=6$, $B=100$. The ranking algorithms are  \texttt{rSVM2}, \texttt{rForest}, \texttt{rBoost} compared to \texttt{MMD}, \texttt{Energy},  \texttt{FR} and \texttt{Tukey}.}
            \label{fig:boxplotH1_L1-}
        \end{figure}

        \begin{figure}[ht!]
            \begin{tabular}{cccc}
                \parbox{0.23\textwidth}{	
                    \includegraphics[width=.2\textwidth]{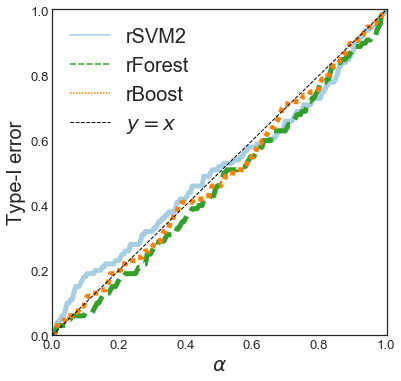}
                    \includegraphics[width=.2\textwidth]{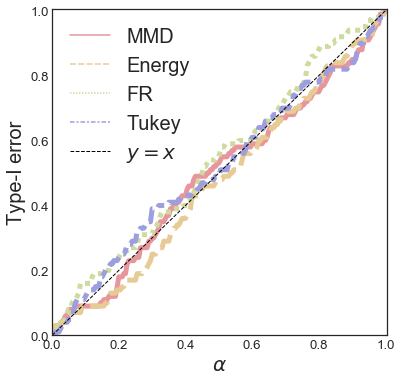}\\
                    {\small  a.  $\cH_0$,  $ \varepsilon=0.0$}
                }
            \parbox{0.23\textwidth}{	
                \includegraphics[width=.24\textwidth]{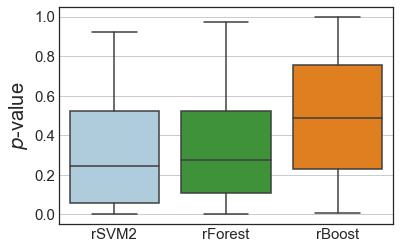}
                \includegraphics[width=.24\textwidth]{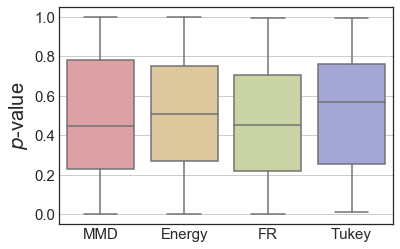}
                {\small b.   $\cH_1$, $ \varepsilon=0.05$}
            }	
            \parbox{0.23\textwidth}{	 
                \includegraphics[width=.23\textwidth]{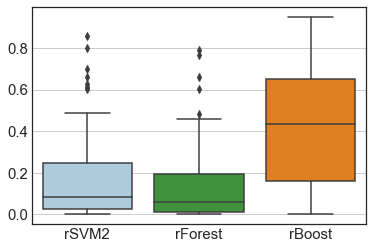}
            \includegraphics[width=.23\textwidth]{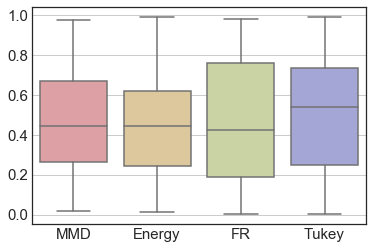}
            {\small c. $\cH_1$,  $ \varepsilon=0.08$}
        }
        \parbox{0.23\textwidth}{	 \includegraphics[width=.23\textwidth]{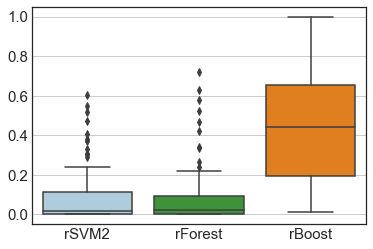}
        \includegraphics[width=.23\textwidth]{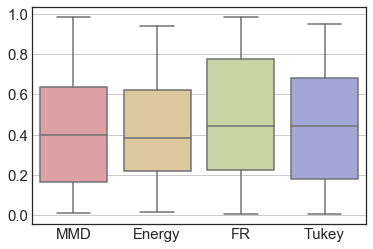}
        {\small d. $\cH_1$,  $ \varepsilon=0.1$}
        }
            \end{tabular}	
        \caption{Rejection rates of the homogeneity assumption  under $\cH_0$ against significance level $\alpha \in (0, 1)$ in Fig. (a). Boxplots of  $p$-values under $\cH_1$ in Fig. (b-d)   for (L1+) with $\phi(u)=u$, $n=m=1000$, $d=6$, $B=100$. The ranking algorithms are  \texttt{rSVM2}, \texttt{rForest}, \texttt{rBoost} compared to \texttt{MMD}, \texttt{Energy},  \texttt{FR} and \texttt{Tukey}.}
            \label{fig:boxplotH1_L1+}
        \end{figure}

        \begin{figure}[ht!]
            \begin{tabular}{cccc}
                \parbox{0.23\textwidth}{	
                        \includegraphics[width=.2\textwidth]{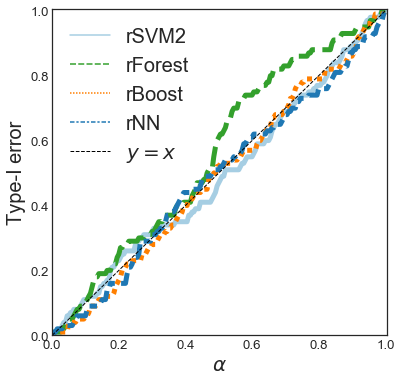}
                    \includegraphics[width=.2\textwidth]{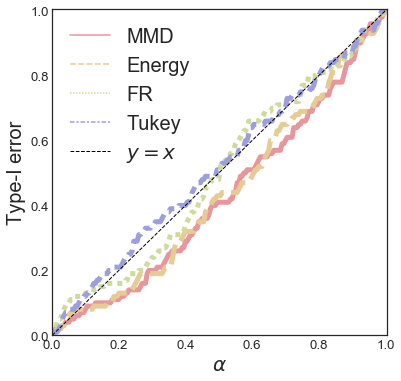}\\
                    {\small  a.  $\cH_0$,  $ \varepsilon=0.0$}
            }
                
                \parbox{0.23\textwidth}{	
            \includegraphics[width=.23\textwidth]{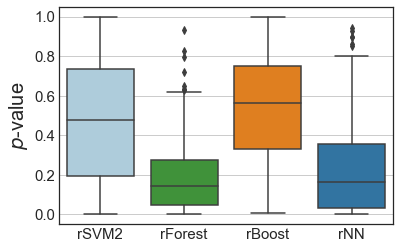}
                    \includegraphics[width=.23\textwidth]{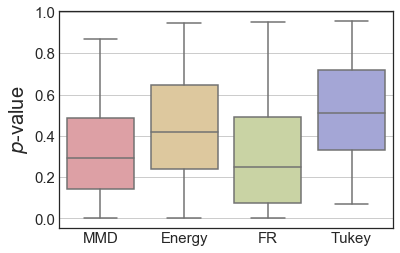}
                    {\small b.  $\cH_1$, $ \varepsilon=0.1$}
                }
            \parbox{0.23\textwidth}{	
            \includegraphics[width=.23\textwidth]{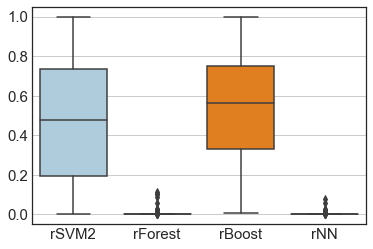}
            \includegraphics[width=.23\textwidth]{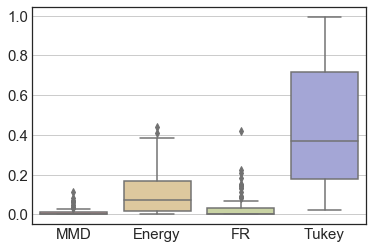}
            {\small c.  $\cH_1$, $ \varepsilon=0.2$}
        }
                
                \parbox{0.23\textwidth}{	
                    \includegraphics[width=.23\textwidth]{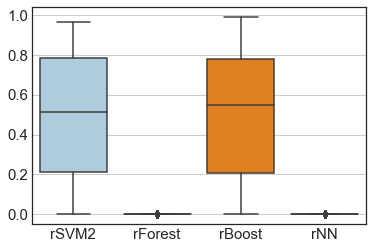}
                    \includegraphics[width=.23\textwidth]{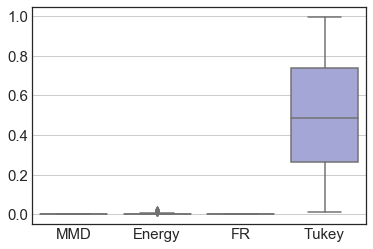}
                    {\small d. $\cH_1$, $ \varepsilon=0.3$}
                }
            \end{tabular}
            \caption{ Rejection rates of the homogeneity assumption  under $\cH_0$ against significance level $\alpha \in (0, 1)$ in Fig. (a). Boxplots of  $p$-values under $\cH_1$ in Fig. (b-d)    for (S1) with $\phi(u)=u$, $n=m=1000$, $d=20$, $B=100$. The ranking algorithms are  \texttt{rSVM2}, \texttt{rForest}, \texttt{rBoost},  \texttt{rNN} compared to \texttt{MMD}, \texttt{Energy},  \texttt{FR} and \texttt{Tukey}.}
            \label{fig:boxplotH1_S2}
        \end{figure}

        \begin{figure}[ht!]
            \begin{tabular}{cccc}
                \parbox{0.23\textwidth}{	
                        \includegraphics[width=.2\textwidth]{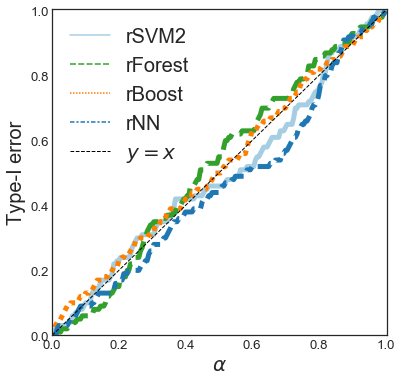}
                    \includegraphics[width=.2\textwidth]{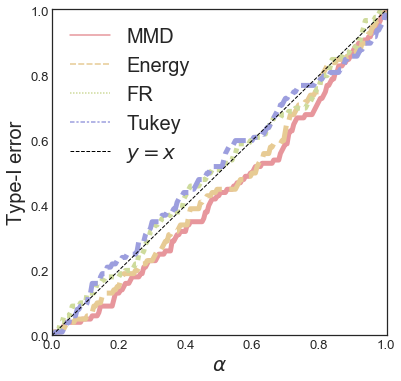}\\
                    {\small  a.  $\cH_0$,  $ \varepsilon=0.0$}
                }
                
                \parbox{0.23\textwidth}{	
                    \includegraphics[width=.24\textwidth]{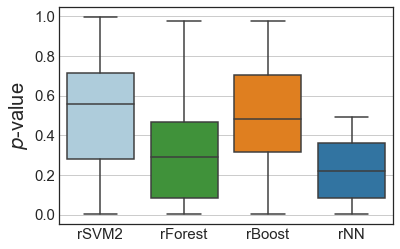}
                    \includegraphics[width=.24\textwidth]{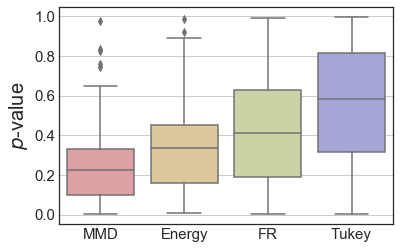}
                    {\small b.  $\cH_1$, $ \varepsilon=0.05$}
                }
                \parbox{0.23\textwidth}{	
                    \includegraphics[width=.23\textwidth]{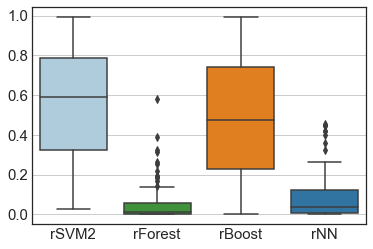}
                    \includegraphics[width=.23\textwidth]{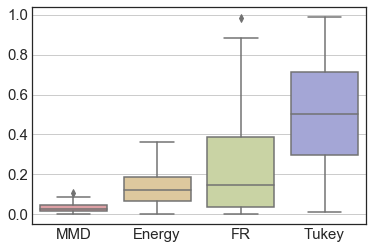}
                    {\small c.  $\cH_1$, $ \varepsilon=0.1$}
                }
                \parbox{0.23\textwidth}{	
                    \includegraphics[width=.23\textwidth]{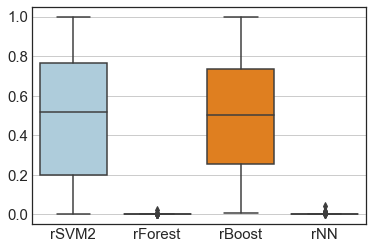}
                    \includegraphics[width=.23\textwidth]{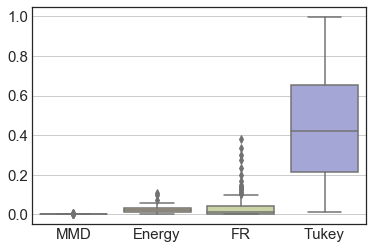}
                    {\small c.  $\cH_1$, $ \varepsilon=0.15$}
                }
            \end{tabular}
            \caption{ Rejection rates of the homogeneity assumption  under $\cH_0$ against significance level $\alpha \in (0, 1)$ in Fig. (a). Boxplots of  $p$-values under $\cH_1$ in Fig. (b-d)    for (S2) with $\phi(u)=u$, $n=m=1000$, $d=30$, $B=100$. The ranking algorithms are  \texttt{rSVM2}, \texttt{rForest}, \texttt{rBoost}, \texttt{rNN} compared to \texttt{MMD}, \texttt{Energy},  \texttt{FR} and \texttt{Tukey}.}
            \label{fig:boxplotH1_S3}
        \end{figure}
        
        \begin{table}[ht!]
            \normalsize
            \renewcommand{\arraystretch}{1.3}
            \centering
            \resizebox{\textwidth}{!}{	
                \begin{tabular}{l||c|ccc||c|ccc}
                    \hline
        
                {} & \multicolumn{4}{c||}{ \textbf{Model (L1-), \quad  $d=6$}  } &\multicolumn{4}{c}{ \textbf{Model (L1+), \quad  $d=6$} }  \\
                \hline
                \textbf{Rejection rate of the null } & \textbf{Under $\cH_0$}&\multicolumn{3}{c||}{ \textbf{Under  $\cH_1$ }  } & \textbf{Under $\cH_0$}&\multicolumn{3}{c}{\textbf{Under $\cH_1$} } \\
                
                \textbf{Method}	&   $\varepsilon=0.0$ & $\varepsilon=0.02$& $\varepsilon=0.05$& $\varepsilon=0.08$& $\varepsilon=0.0$ & $\varepsilon=0.05$& $\varepsilon=0.08$& $\varepsilon=0.1$\\
                    \hline 
                    rSVM2&0.11 $\pm$ 0.31 &\bf 0.22 $\pm$ 0.49 & 0.60  $\pm$ 0.49 & 0.96 $\pm$ 0.20 & 0.10 $\pm$ 0.30 &\bf 0.23 $\pm$ 0.42&0.37 $\pm$ 0.48&\bf 0.66 $\pm$ 0.48\\
                    \bf rForest& \bf 0.05   $\pm$ 0.22 &0.15 $\pm$ 0.36 & \bf 0.71 $\pm$ 0.46 &\bf 0.98 $\pm$  0.14 &\bf 0.05   $\pm$ 0.22  & 0.19    $\pm$ 0.39 &\bf 0.48 $\pm$ 0.50& 0.62 $\pm$ 0.49 \\			
                    rBoost &  0.09 $ \pm$ 0.29 & 0.07 $\pm$ 0.26 & 0.08$\pm$ 0.27 & 0.11  $\pm$ 0.31  & 0.06 $\pm$ 0.24 &0.10$\pm$ 0.30 & 0.07 $\pm$ 0.26& 0.06 $\pm$ 0.24\\			
                    \hline
                    MMD& \bf 0.05 $\pm$ 0.22 &0.04 $\pm$ 0.20&0.03 $ \pm$ 0.17&0.04 $\pm$ 0.20 & 0.08 $\pm$ 0.27&0.06 $\pm$ 0.24&0.06 $\pm$ 0.24& 0.06 $\pm$ 0.24 \\			
                    Energy&0.06 $\pm$ 0.24& 0.05 $\pm$ 0.22&   0.03 $ \pm$ 0.17 &0.06 $\pm$ 0.24&0.06 $\pm$ 0.24& 0.05 $\pm$ 0.22 &0.04 $\pm$ 0.20 & 0.02 $\pm$ 0.14 \\			
                    FR &0.09 $ \pm$ 0.29 &0.09 $\pm$ 0.29   &0.09 $ \pm$ 0.29 &0.06 $\pm$ 0.24 & 0.08 $\pm$ 0.27&0.10$\pm$ 0.30 &0.09 $ \pm$ 0.29 &  0.07 $\pm$ 0.26 \\
                    Tukey &0.06 $\pm$ 0.24& 0.05 $\pm$ 0.22 &   0.08 $ \pm$ 0.27 &0.07 $\pm$ 0.26&\bf  0.05 $\pm$ 0.22  &0.08 $\pm$ 0.27 &0.07 $\pm$ 0.26&0.09 $ \pm$ 0.29  \\
                    
                \hline
                {} & \multicolumn{4}{c||}{ \textbf{Model (S1), \quad $d=20$} } &\multicolumn{4}{c}{ \textbf{Model (S2), \quad $d=30$}  }  \\
                \hline
                \textbf{Rejection rate of the null } & \textbf{Under $\cH_0$}&\multicolumn{3}{c||}{ \textbf{Under  $\cH_1$ }  } & \textbf{Under $\cH_0$}&\multicolumn{3}{c}{\textbf{Under $\cH_1$} } \\
                \textbf{Method}	&   $\varepsilon=0.0$ & $\varepsilon=0.1$& $\varepsilon=0.2$& $\varepsilon=0.3$&   $\varepsilon=0.0$ & $\varepsilon=0.05$& $\varepsilon=0.1$& $\varepsilon=0.15$\\
                    \hline 
                    rSVM2&0.06  $\pm$ 0.24 &0.03$\pm$ 0.17 &0.06 $\pm$ 0.24& 0.09 $ \pm$ 0.29 &0.04 $\pm$ 0.20&0.04 $\pm$ 0.20&0.04 $\pm$ 0.20&0.04 $\pm$ 0.20 \\
                    \bf	rForest & 0.05 $\pm$ 0.22  & \bf 0.30 $\pm$ 0.46   &0.95 $\pm$ 0.22&  \bf 1.00 &\bf 0.03 $\pm$ 0.17 & 0.19 $\pm$ 0.39 &  0.72 $\pm$ 0.45& \bf 1.00 \\
                    rBoost& 0.04 $\pm$ 0.20   &  0.02 $\pm$ 0.14& 0.08 $ \pm$ 0.27 & 0.10 $\pm$ 0.30  &0.09 $\pm$ 0.29&0.04 $\pm$ 0.20&0.05 $\pm$ 0.22& 0.03 $ \pm$ 0.17 \\
                    rNN& \bf 0.03 $\pm$ 0.17 & 0.27 $\pm$ 0.45 &\bf 0.98 $\pm$  0.14 &\bf 1.00 & 0.04 $\pm$ 0.20  &\bf 0.20  $ \pm$ 0.40 & 0.54 $\pm$ 0.50 &\bf 1.00 \\
                    \hline
                    MMD & 0.04 $\pm$ 0.20 &0.08 $ \pm$ 0.27  &0.93 $\pm$ 0.26 &\bf 1.00&0.04 $\pm$ 0.20&  0.13 $\pm$ 0.34&\bf 0.79 $\pm$ 0.41&\bf 1.00 \\
                    Energy& 0.06  $\pm$ 0.24 &0.07 $\pm$ 0.26 &0.39 $\pm$ 0.49&\bf1.00& 0.04 $\pm$ 0.20& 0.11 $\pm$ 0.31&0.18$\pm$ 0.39&0.93 $\pm$ 0.26\\
                    
                    FR &0.10 $\pm$ 0.30 &0.19 $\pm$ 0.39 &0.82 $\pm$ 0.39&\bf1.00& 0.06 $\pm$ 0.24 &0.06 $\pm$ 0.24&0.28$\pm$ 0.45&0.78 $\pm$ 0.42 \\
                    
                    Tukey &  0.08 $ \pm$ 0.27 &0.00 &0.02 $\pm$ 0.14&0.04 $\pm$0.20 & 0.04 $\pm$ 0.20&0.04 $\pm$ 0.20 &0.04$\pm$ 0.20 & 0.05 $\pm$ 0.22  \\
                    \hline
                \end{tabular}
                }
            \caption{ Rejection rates of the homogeneity assumption under $\cH_0$ and $\cH_1$ at significance level $\alpha=0.05$ for the location models (L1-) and (L1+) and for the scale models (S1) and (S2), with $n = m=1000$, $\pm$ their standard deviation at $95\%$.  For ranking methods, only the results associated to MWW test are presented. Methods in bold minimize among all algorithms the rejection rate under $\cH_0$ while maximizing under $\cH_1$.  }
            \label{tab:powerL}
        \end{table}

        \begin{figure}[ht!]
            \begin{tabular}{cccc}
                \parbox{0.23\textwidth}{	
                    \includegraphics[width=.2\textwidth]{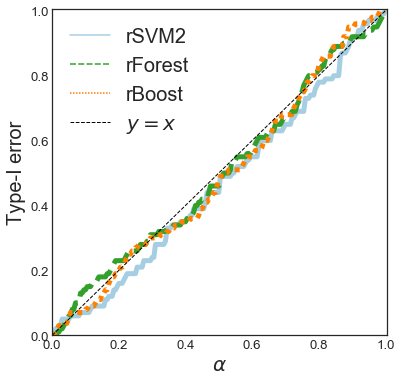}
                    \includegraphics[width=.2\textwidth]{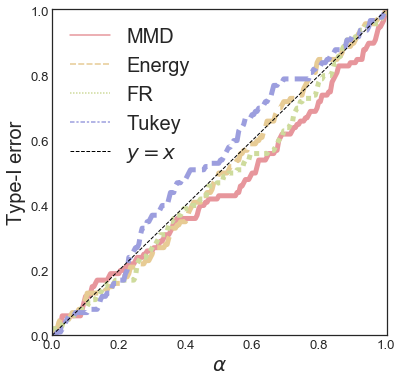}\\
                    {\small  a.  $\cH_0$,  $ \varepsilon=0.0$}
                }
            \parbox{0.23\textwidth}{	
                \includegraphics[width=.24\textwidth]{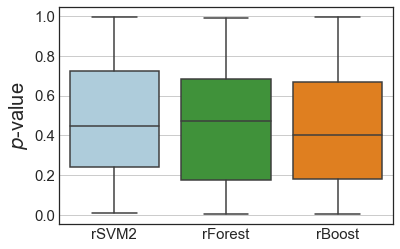}
                \includegraphics[width=.24\textwidth]{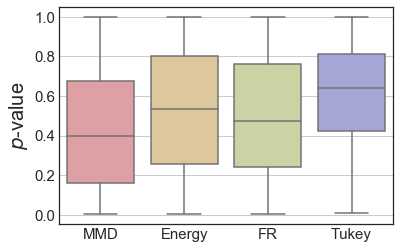}
                {\small b. $\cH_1$, $\varepsilon=0.05$}
            }
            \parbox{0.23\textwidth}{	
            \includegraphics[width=.23\textwidth]{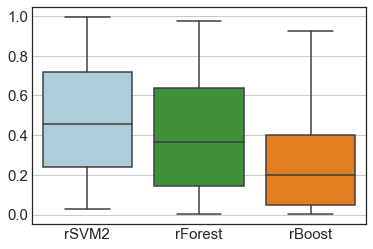}
            \includegraphics[width=.23\textwidth]{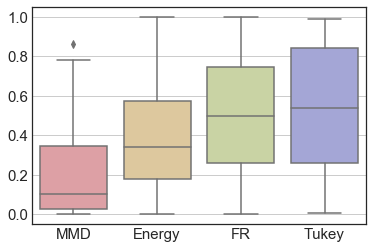}
            {\small c. $\cH_1$,  $\varepsilon=0.15$}
        }
            \parbox{0.23\textwidth}{	
            \includegraphics[width=.23\textwidth]{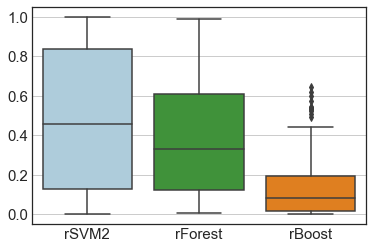}
            \includegraphics[width=.23\textwidth]{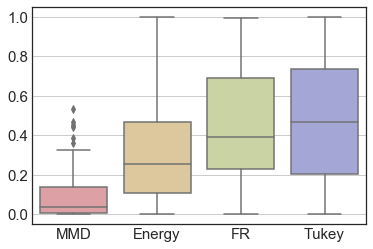}
            {\small d. $\cH_1$,   $\varepsilon=0.20$}
        }	
            \end{tabular}
            \caption{ Rejection rates of the homogeneity assumption  under $\cH_0$ against significance level $\alpha \in (0, 1)$ in Fig. (a). Boxplots of  $p$-values under $\cH_1$ in Fig. (b-d)    for (T1) with $\phi(u)=u$, $n=m=1000$, $d=3$, $B=100$. The ranking algorithms are  \texttt{rSVM2}, \texttt{rForest}, \texttt{rBoost} compared to \texttt{MMD}, \texttt{Energy},  \texttt{FR} and \texttt{Tukey}.}
            \label{fig:boxplotH1_T1}
        \end{figure}

        \begin{figure}[ht!]
            \begin{tabular}{cccc}
                
                \parbox{0.23\textwidth}{	
                    \includegraphics[width=.2\textwidth]{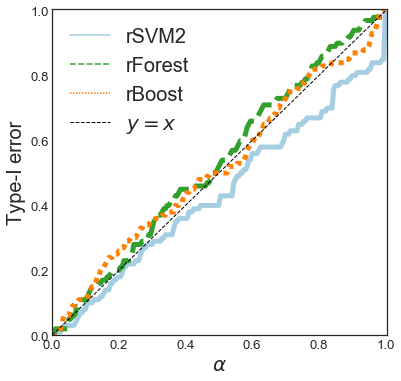}
                \includegraphics[width=.2\textwidth]{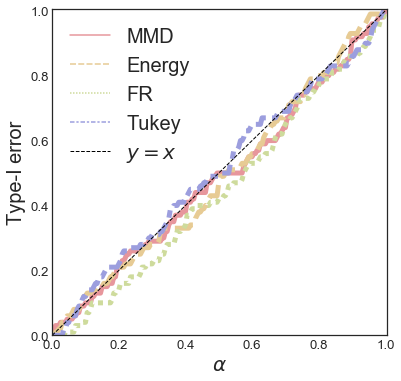}\\
                {\small  a.  $\cH_0$,  $ \varepsilon=0.0$}
                }	
            
                \parbox{0.23\textwidth}{	
                    \includegraphics[width=.24\textwidth]{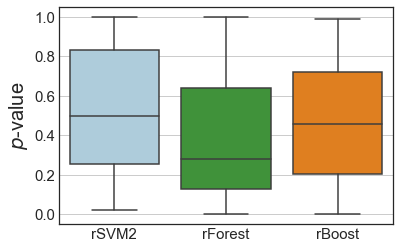}
                \includegraphics[width=.24\textwidth]{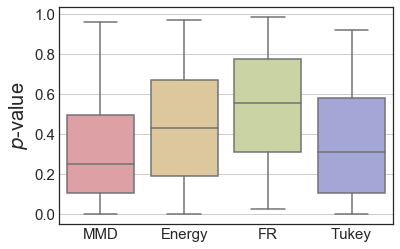}
                    {\small b. $\cH_1$, $ \varepsilon=0.1$}
                }	
        \parbox{0.23\textwidth}{	
            \includegraphics[width=.23\textwidth]{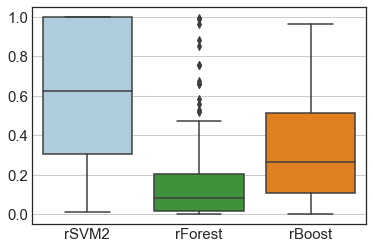}
        \includegraphics[width=.23\textwidth]{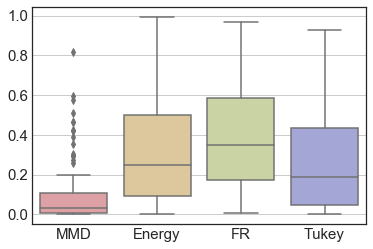}
            {\small c. $\cH_1$,  $ \varepsilon=0.2$}
        }\parbox{0.23\textwidth}{	
        \includegraphics[width=.23\textwidth]{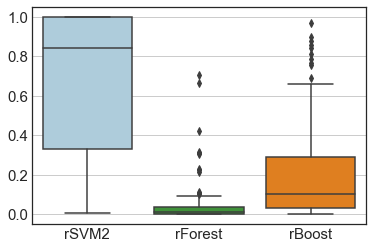}
        \includegraphics[width=.23\textwidth]{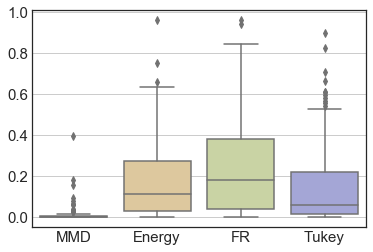}
        {\small d. $\cH_1$, $ \varepsilon=0.3$}
        }
            \end{tabular}
            \caption{ Rejection rates of the homogeneity assumption  under $\cH_0$ against significance level $\alpha \in (0, 1)$ in Fig. (a). Boxplots of  $p$-values under $\cH_1$ in Fig. (b-d)   for (T2) with (L1-) with $\phi(u)=u$, $n=m=1000$, $d=4$, $B=100$. The ranking algorithms are  \texttt{rSVM2}, \texttt{rForest}, \texttt{rBoost} compared to \texttt{MMD}, \texttt{Energy},  \texttt{FR} and \texttt{Tukey}.}
            \label{fig:boxplotH1_T2}
        \end{figure}

        \begin{figure}[ht!]
            \begin{tabular}{cccc}
                
                \parbox{0.23\textwidth}{
                        \includegraphics[width=.2\textwidth]{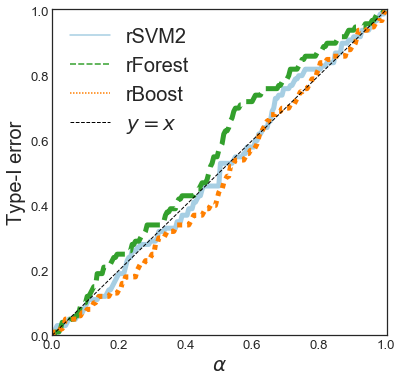}
                    \includegraphics[width=.2\textwidth]{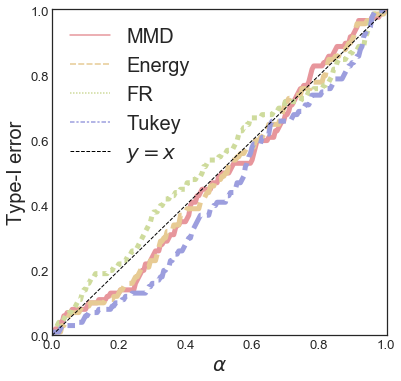}\\
                    {\small  a.  $\cH_0$,  $ \varepsilon=0.0$}
                }	
                \parbox{0.23\textwidth}{	
                    \includegraphics[width=.24\textwidth]{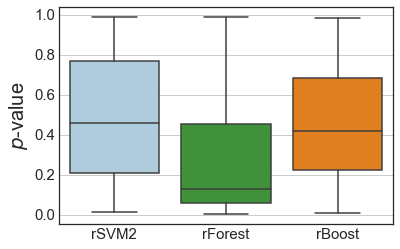}
                    \includegraphics[width=.24\textwidth]{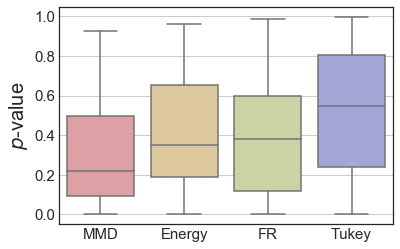}
                    {\small b. $\cH_1$,  $ \varepsilon=0.1$}
                }	
                \parbox{0.23\textwidth}{	
                    \includegraphics[width=.23\textwidth]{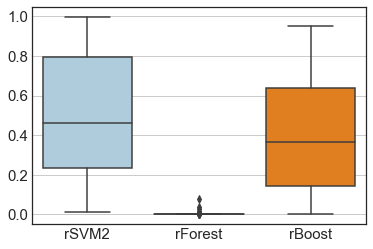}
                    \includegraphics[width=.23\textwidth]{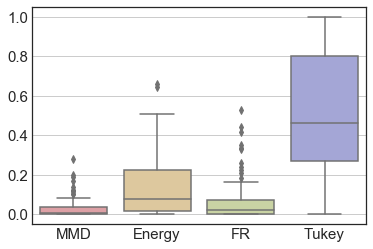}
                    {\small c. $\cH_1$, $ \varepsilon=0.2$}
                }
            \parbox{0.23\textwidth}{	
            \includegraphics[width=.23\textwidth]{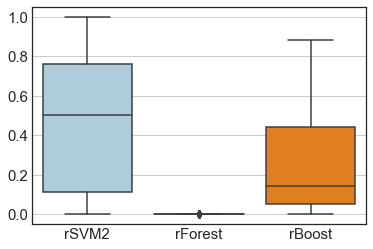}
            \includegraphics[width=.23\textwidth]{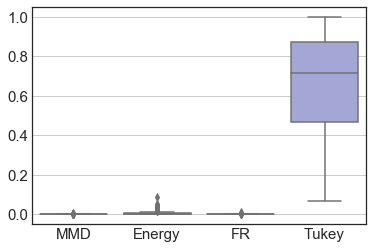}
            {\small d. $\cH_1$,   $ \varepsilon=0.3$}
        }	
            \end{tabular}
            \caption{ Rejection rates of the homogeneity assumption  under $\cH_0$ against significance level $\alpha \in (0, 1)$ in Fig. (a). Boxplots of  $p$-values under $\cH_1$ in Fig. (b-d)  for (T3) with (S1) with $\phi(u)=u$, $n=m=1000$, $d=20$, $B=100$. The ranking algorithms are  \texttt{rSVM2}, \texttt{rForest}, \texttt{rBoost} compared to \texttt{MMD}, \texttt{Energy},  \texttt{FR} and \texttt{Tukey}.}
            \label{fig:boxplotH1_T3}
        \end{figure}

        \begin{table*}[ht!]
            \tiny 
            \centering
            \resizebox{\textwidth}{!}{	
                \begin{tabular}{l||c|ccc}
                    \hline
                    {} & \multicolumn{4}{c}{ \textbf{Model (T1),  \quad  $d=3$} }  \\
                    \hline
                    \textbf{Rejection rate of the null } & \textbf{Under $\cH_0$} &\multicolumn{3}{c}{\textbf{Under $\cH_1$  }  } \\
                    \textbf{Method}	& $\varepsilon=0.0$  & $\varepsilon=0.05$& $\varepsilon=0.15$& $\varepsilon=0.20$\\
                    
                    \hline 
                    rSVM2 &\bf 0.05 $\pm$ 0.22 &0.06 $\pm$ 0.24&0.06 $\pm$ 0.24 &0.14 $\pm$ 0.35  \\
                    rForest &0.06 $\pm$ 0.24    &\bf 0.08 $\pm$ 0.27& 0.14 $\pm$ 0.35  &0.17 $\pm$ 0.38 \\
                    rBoost &\bf 0.05 $\pm$ 0.22   &0.06 $\pm$ 0.24 &0.26  $\pm$ 0.44 &0.41 $\pm$ 0.5 \\
                    \hline
                    \bf MMD&   0.06 $\pm$ 0.24&0.04 $\pm$ 0.20 &\bf 0.35   $\pm$ 0.48&\bf0.57   $\pm$ 0.50\\
                    Energy&\bf0.05 $\pm$ 0.22   & 0.05 $\pm$ 0.22 &0.09 $\pm$ 0.29&0.15 $\pm$ 0.36 \\
                    FR & \bf0.05 $\pm$ 0.22  &  0.05 $\pm$ 0.22  &0.08 $\pm$ 0.27&0.06 $\pm$ 0.24 \\
                    Tukey &  0.06 $\pm$ 0.24 &0.07 $\pm$ 0.26  &0.04 $\pm$ 0.20&0.07 $\pm$ 0.26  \\
                    
                    \hline
                        {} & \multicolumn{4}{c}{ \textbf{Model (T2), \quad  $d=4$} }  \\
                    \hline
                    \textbf{Rejection rate of the null } & \textbf{Under $\cH_0$} &\multicolumn{3}{c}{\textbf{Under $\cH_1$  }  } \\
                    \textbf{Method }	& $\varepsilon=0.0$ & $\varepsilon=0.10$& $\varepsilon=0.20$& $\varepsilon=0.30$\\
        
                    \hline 
                    rSVM2 &0.03 $\pm$ 0.17&0.06 $\pm$ 0.24 &0.04 $\pm$ 0.20 & 0.03 $\pm$ 0.17 \\
                    rForest &0.05 $\pm$ 0.22 &0.12 $\pm$ 0.33 & 0.43$\pm$ 0.50& 0.80 $\pm$ 0.40\\
                    rBoost & 0.06 $\pm$ 0.24 & 0.09 $\pm$ 0.29 & 0.18 $\pm$ 0.39 & 0.34 $\pm$ 0.48  \\
                    \hline
                \bf	MMD&0.05 $\pm$ 0.22 &\bf 0.18 $\pm$ 0.39 &\bf 0.56  $\pm$ 0.50&\bf 0.92 $\pm$ 0.27   \\
                    Energy&0.06 $\pm$ 0.24  &0.08 $\pm$ 0.27 & 0.16  $\pm$ 0.37 &0.32$\pm$ 0.47  \\
                    FR &\bf 0.01 $\pm$ 0.10 &0.02 $\pm$ 0.14 & 0.08 $\pm$ 0.26&0.28 $\pm$ 0.45\\
                    Tukey &0.04 $\pm$ 0.20 &0.12 $\pm$ 0.33 & 0.27  $\pm$ 0.45&0.46 $\pm$ 0.50  \\
                    
                    \hline
                    {} & \multicolumn{4}{c}{ \textbf{Model (T3), \quad $d=20$} }  \\
                \hline
                \textbf{Rejection rate of the null } & \textbf{Under $\cH_0$} &\multicolumn{3}{c}{\textbf{Under $\cH_1$  }  } \\
                    \textbf{Method}	& $\varepsilon=0.0$  & $\varepsilon=0.1$& $\varepsilon=0.2$& $\varepsilon=0.3$ \\
                        \hline 
                        rSVM2 &0.05 $\pm$ 0.22 &0.06 $\pm$ 0.24  &0.05 $\pm$ 0.22 &0.13 $\pm$ 0.34 \\
                        \bf	rForest & 0.04 $\pm$ 0.20&\bf 0.23 $\pm$ 0.42 &\bf 0.99 $\pm$ 0.10 & \bf 1.00 \\
                        rBoost &0.05 $\pm$ 0.22 &0.07 $\pm$ 0.26 & 0.11 $\pm$ 0.31 &0.24  $\pm$ 0.43\\
                        \hline
                        MMD&  0.07 $\pm$  0.26& 0.15 $\pm$ 0.36  & 0.83 $\pm$ 0.38 & 1.00\\
                        Energy& 0.06 $\pm$ 0.24&0.09 $\pm$ 0.29&0.40 $\pm$ 0.49&0.98 $\pm$ 0.14 \\
                        FR & 0.07 $\pm$  0.26& 0.09 $\pm$ 0.29&0.63 $\pm$ 0.49 &1.00\\
                        Tukey & \bf 0.03 $\pm$  0.17 & 0.05 $\pm$ 0.22 &0.04 $\pm$ 0.20& 0.00\\
                        \hline
                \end{tabular}
            }		
            \caption{Rejection rates of the homogeneity assumption under $\cH_0$ and $\cH_1$ at significance level $\alpha=0.05$ for models (T1-3), with $n = m=1000$, $\pm$ their standard deviation at $95\%$.  For ranking methods, only the results associated to MWW test are presented. Methods in bold minimize among all algorithms the rejection rate under $\cH_0$ while maximizing under $\cH_1$.}
            \label{tab:powerST}
        \end{table*}

        \begin{figure}[ht!]
            \begin{tabular}{ccc}
                \parbox{0.3\textwidth}{	
                    \includegraphics[width=.3\textwidth]{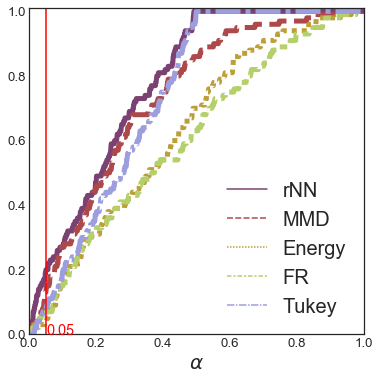}
                    {\small a.  $d = 30$,  $\varepsilon= 0.05$}\\
                        \includegraphics[width=.3\textwidth]{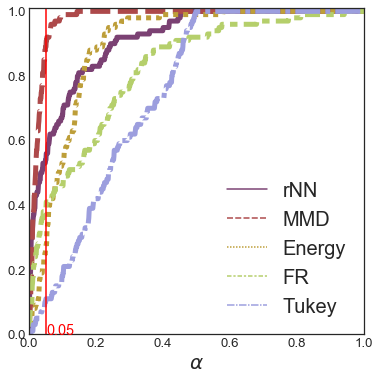}\\
                    {\small e.  $d = 30$, $\varepsilon=0.1$}
                }	
                
                \parbox{0.3\textwidth}{	
                    \includegraphics[width=.3\textwidth]{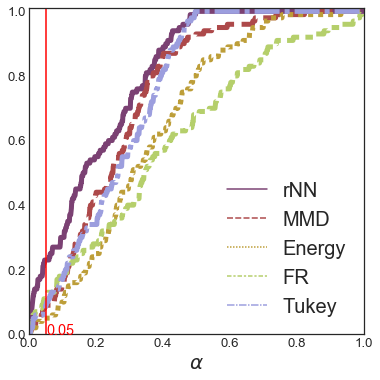}
                    {\small b.  $d = 60$,  $\varepsilon= 0.05$}\\
                        \includegraphics[width=.3\textwidth]{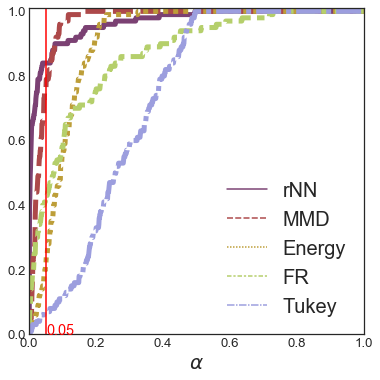}\\
                        {\small f.  $d = 60$,  $\varepsilon= 0.1$}
                }	
                \parbox{0.3\textwidth}{	
                    \includegraphics[width=.3\textwidth]{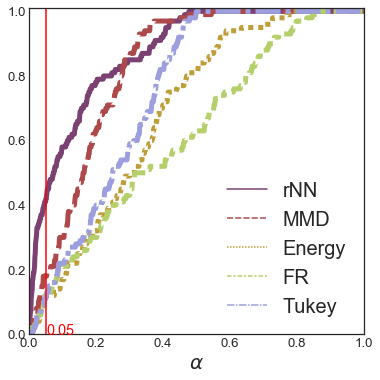}
                    {\small c.  $d = 100$,  $\varepsilon= 0.05$}\\
                    \includegraphics[width=.3\textwidth]{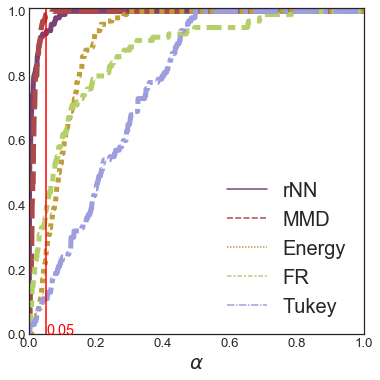}\\
                    {\small g.  $d = 100$,  $\varepsilon= 0.1$}
                }	
            \end{tabular}
            \caption{Rejection rates of the homogeneity assumption under $\cH_1$ against significance level $\alpha \in (0, 1)$ for increasing dimension $d \in \{30, 60, 100\}$ of the feature space for (S2) with $\phi(u)=u$, $n=m=1000$, $B=100$ and $\varepsilon= 0.05$ (a-c), $\varepsilon=0.1$ (e-g). The ranking algorithm  is   \texttt{rNN} compared to \texttt{MMD}, \texttt{Energy},  \texttt{FR} and \texttt{Tukey}.}
            \label{fig:power_size_S3}
        \end{figure}

        \begin{table*}[ht!]
            \renewcommand{\arraystretch}{1.1}
            \centering
            \resizebox{\textwidth}{!}{	
                \begin{tabular}{l||ccc|ccc}
                    \hline
                \textbf{Rejection rate of the null } &\multicolumn{3}{c|}{\textbf{Under $\cH_1$,  \quad ($\varepsilon = 0.05$)}   } & \multicolumn{3}{c}{\textbf{Under $\cH_1$,  \quad ($\varepsilon = 0.1$)}  } \\
                    \textbf{Method}	& $d=30$& $d=60$& $d=100$&  $d=30$& $d=60$& $d=100$\\
                    
                    \hline 
                    \bf rNN&\bf0.20 $\pm$ 0.40 &\bf0.23 $\pm$  0.42&\bf0.42  $\pm$ 0.50 &0.54  $\pm$ 0.50&\bf0.84$\pm$ 0.37 &0.93 $\pm$ 0.26\\
                    \hline
                    MMD&0.13 $\pm$ 0.34 & 0.09 $\pm$ 0.29 &0.18 $\pm$  0.39&\bf0.89  $\pm$ 0.31&0.76 $\pm$ 0.43&\bf 0.99 $\pm$ 0.10\\
                    Energy&0.04 $\pm$ 0.20&0.05 $\pm$ 0.22 &0.11 $\pm$  0.31&0.27  $\pm$ 0.45&0.22 $\pm$ 0.42 &0.23 $\pm$ 0.42\\
                    FR & 0.06 $\pm$ 0.24 &0.11 $\pm$0.31&0.13 $\pm$  0.34&0.39  $\pm$ 0.49&0.43 $\pm$ 0.50&0.38 $\pm$ 0.49 \\
                    Tukey & 0.08 $\pm$ 0.27 &0.09 $\pm$ 0.29 &0.11 $\pm$  0.31 &0.11 $\pm$ 0.31 &0.07 $\pm$ 0.26& 0.11 $\pm$ 0.31\\	
                    \hline
                \end{tabular}
            }		
            \caption{Rejection rates of the homogeneity assumption under $\cH_1$ at significance level $\alpha=0.05$ for the scale model (S2), with $n = m=1000$, $\pm$ their standard deviation at $95\%$.  For \texttt{rNN}, only the results associated to MWW test are presented. Methods in bold maximize among all algorithms the rejection rate under $\cH_1$.     }
            \label{tab:powersizeS3}
        \end{table*}

        \section{Conclusion}\label{sec:conclusion}
        In this paper, we have developed a fully novel and flexible approach to the two-sample problem, ubiquitous in statistical and machine-learning applications.\\ 
        It relies on the observation that the null/homogeneity assumption $\mathcal{H}_0$ ceases to hold true as soon as the optimal $\roc$ curve for the bipartite ranking problem related to the probability distributions $H$ and $G$ of the two samples observed does not coincide with the diagonal (and is then necessarily above it at some points of the $\roc$ space). As the curve $\roc^*_{H,G}$ is not directly observable, the methodology proposed involves a statistical learning step to solve the bipartite ranking problem based on a first fraction of the samples. It next consists in evaluating the departure from $\mathcal{H}_0$ using an empirical version of (a scalar summary of) $\roc^*_{H,G}$. In practice, the second step simply boils down to performing a classic two-sample rank test based on the ranks of the fraction of the samples not used in the learning stage, obtained by means of the ranking rule constructed in the first step. \\
        A sound theoretical analysis has been carried out, showing that the two types of errors made by the method can be controlled in a nonasymptotic fashion. It is complemented by an experimental study providing strong empirical evidence of the merits of the ranking-based method compared to some popular alternatives: in general, it exhibits a performance that better withstands high-dimensionality, better resists as one gets closer and closer to the null hypothesis and adapts to a  wide variety of ways of departing from homogeneity. However, several issues remain to be investigated so that the general approach introduced here can systematically perform well in practice. \\
        One may state two open problems in particular. Understanding how to dispatch the data at disposal, for the statistical learning purpose and for performing the rank test, as well as developing a method to combine efficiently several ranking-based rank tests so as to provably improve performance (insofar as one does not know in advance which score generating function $\phi$ may capture best the possible type of departure from homogeneity exhibited), requires further research.



\bibliographystyle{abbrvnat}
\bibliography{Ref_jmlr}

\appendix

	\section{Properties of the $\roc$ curves}\label{app:rocprop}
For the sake of completeness, basic properties of $\roc$ curves are recalled. See \textit{e.g.} Proposition 17 in \cite{CV09ieee}.

\begin{proposition}{\sc(Properties of the $\roc$ curve)}\label{prop:properties}
For any probability distributions $H$ and $G$ and any scoring function $s : \Z \to
\RR$, the following properties hold:
\begin{enumerate}
\item {\bf Limit values.} We have: $$\roc(s, 0) =0 \; \text{ and } \; \roc(s, 1) =1~.$$
\item {\bf Invariance.} For any strictly increasing function $T  : \RR \to
\RR$, we have, for all $\alpha\in (0,1)$: $\roc (T \circ s, \alpha)
= \roc( s, \alpha)$.
\item {\bf Concavity.} If the likelihood ratio $\dd G_s/\dd H_s$ is a
monotone function then the $\roc$ curve is concave. In particular, the optimal curve $\roc^*$ is concave.
\item {\bf Linear parts.} If the likelihood ratio $\dd G_s/\dd H_s$ is
constant on some interval in the range of the scoring function $s$
then the $\roc$ curve will present a linear part on the
corresponding domain. Furthermore, $\roc^*$ is linear on $[\alpha_1,\alpha_2]$ iff $\Psi(z)=(\dd G/\dd H)(z)$ is constant on the subset $$\{z\in \mathcal{Z}/\; Q^*(\alpha_2)\leq \Psi(z) \leq Q^*(\alpha_1)\}~,$$ where $Q^*(\alpha)$ means the quantile at level $1-\alpha\in (0,1)$ of the $\rv$ $\Psi(\mathbf{Y})$.
\item {\bf Differentiability.} Assume that the distributions $H$ and $G$ are continuous. Then, the $\roc$ curve of a scoring
function $s$ is differentiable $\IFF$ the distributions $H_s$ and $G_s$ are continuous.
\end{enumerate}
\end{proposition}
	
	\section{Technical Proofs} \label{app:proofs}
	\subsection{Proof of Formula \eqref{eq:AUC_connect}}
	Given two real-valued, independent and continuous random variables $X\sim G$ and $Y\sim H$, the proof is straightforward by noticing that:
	\begin{equation*}
		\int_{-\infty}^{\infty}\{H(t)-G(t)\}\dd H(t) = \frac{1}{2} - \int_{-\infty}^{\infty}G(t)\dd H(t)~.
	\end{equation*}
	Hence, we have:
	\begin{multline*}
		\auc_{H,G}= \mathbb{P}\{Y \leq X   \} = \EE[ \EE[\mathbb{I}\{Y \leq X   \}]  \mid Y]
		 = \EE_{Y\sim H}[1-G(Y)] \\
		 = \frac{1}{2} +	\int_{-\infty}^{\infty}\{H(t)-G(t)\}\dd H(t)~.
	\end{multline*}

	\subsection{Proof of Proposition \ref{prop:rationale}}
	The equivalence between assertions $(i)$ and $(ii)$ can be 
directly deduced from the following result, proved in  \cite{CV09ieee} (see Corollary 5 and Proposition 6's proof therein), recalled here for the sake of clarity.

\begin{lemma} (\cite{CV09ieee})
It holds with probability one:
\begin{equation*}
\Psi(\bZ)=\frac{\dd G_{\Psi}}{\dd H_{\Psi}}(\Psi(\bZ))~,
\end{equation*}
$\bZ$ denoting either $\bX$ or $\bY$.
\end{lemma}

The equivalence of assertion $(ii)$ with the other assertions is immediate using formula \eqref{eq:W_crit}.
	
	\subsection{An Exponential Inequality for Two-Sample Linear Rank Statistics}
	The result stated below provides a tail bound of exponential type for a two-sample linear rank statistic (see Definition \ref{def:R_stat}), when recentered by its asymptotic mean \eqref{eq:lim_mean}.

\begin{theorem}\label{thm:tail}
Let $p\in (0,1)$ and $N\geq 1/p$. Set $n = \lfloor pN \rfloor$ and $m = \lceil (1-p)N \rceil = N - n$. Let  $\{X_1,\;\ldots,\; X_n\}$ and  $\{Y_1,\;\ldots,\; Y_m\}$ be two independent i.i.d. random samples, drawn from univariate probability distributions $G$ and $H$ respectively. Suppose that Assumption \ref{hyp:phic2} is fulfilled. Then, for all $t>0$, we have:
\begin{equation}
\mathbb{P}\left\{\frac{1}{n}\widehat{W}_{n,m}^{\phi} -W_{\phi}>t \right\}\leq 18 \exp\{-CNt^2\}~,
\end{equation}where 
$C=8^{-1}\min\left(p/\lVert \phi \rVert_{\infty}^2,(p \lVert \phi' \rVert_{\infty}^2)^{-1},((1-p) \lVert \phi' \rVert_{\infty}^2)^{-1} \right) $.
\end{theorem}

\begin{proof}
The proof is based on the linearization technique used to establish the concentration inequality for two-sample linear rank processes in \cite{CleLimVay21}, see Theorem 5 therein.
Consider $\phi$ fulfilling Assumption \ref{hyp:phic2}. Summing $\phi$'s Taylor expansions of order $2$, evaluated at $ N\widehat{F}_{N}(X_i)/(N+1)$ around $\bF(X_i)$ for $1\leq i\leq n$ leads to an \textit{a.s.} decomposition of the statistic $\widehat{W}_{n,m}^{\phi}$, see Eq. (B.3,4) in \cite{CleLimVay21}. 
After the application of Hoeffding's decomposition to its first order term, which takes the form of a two-sample $U$-statistic, we write the decomposition as follows:
\begin{multline}\label{eq:linear}
		\frac{1}{n}\widehat{W}_{n,m}^{\phi} - W_{\phi} = \widehat{W}_{\phi}-  W_{\phi} +\frac{1}{n} \left(\widehat{V}_{n}^X -\mathbb{E}\left[ \widehat{V}_{n}^X \right]\right)
		+\frac{1}{n}\left( \widehat{V}_{m}^Y -\mathbb{E}\left[  \widehat{V}_{m}^Y \right]\right)
	+\frac{1}{n} \mathcal{R}_{n,m}~,
\end{multline}
where
\begin{eqnarray*}
	\widehat{W}_{\phi}&=& \frac{1}{n} \sum_{i=1}^{n} \left(\phi \circ \bF\right)(X_i)~, \\
	\widehat{V}_{n}^X &=& \frac{n-1}{N+1} \sum_{i=1}^n \int_{X_i}^{+\infty} (\phi' \circ \bF)(u) \dd G(u)~,\\
	\widehat{V}_{m}^Y &=& \frac{n}{N+1} \sum_{j=1}^m \int_{Y_j}^{+\infty}( \phi' \circ \bF)(u) \dd G(u)~.
\end{eqnarray*}
The quantity $ \mathcal{R}_{n,m}$ corresponds to the Taylor-Lagrange residual term plus additional statistics of order $\mathcal{O}_{\mathbb{P}}(N^{-1})$ inherited from  Hoeffding's decomposition. The lemma below provides a tail bound for the latter.

\begin{lemma}\label{lem:cbrem} Suppose that the assumptions of Theorem \ref{thm:tail} are satisfied. Then, for all $t>0$ and $N\geq2$, we have, if $Nt \geq 128  \lVert \phi' \rVert_{\infty}^2/( p  \lVert \phi'' \rVert_{\infty} ) $:
	\begin{equation}
		\mathbb{P}\left\{ \vert \mathcal{R}_{n,m} \vert > t  \right\} \leq 12 \exp\left\{- \frac{Nt}{12\kappa_p \lVert \phi'' \rVert_{\infty}}\right\}~,
	\end{equation}
	and otherwise
	\begin{equation}
		\mathbb{P}\left\{ \vert \mathcal{R}_{n,m} \vert > t  \right\} \leq 12 \exp\left\{- \frac{\alpha_p N^2t^2}{512 \lVert \phi' \rVert_{\infty}^2}\right\}~,
	\end{equation}
	where $\alpha_p = \min(p /(1-p),1)$, $\kappa_p = \max(p, 1-p)$.
\end{lemma}

For any $t>0$, a straightforward application of the classic exponential tail inequality, proved in  \cite{Hoeffding63}, to each of the first three terms on the right hand side of \eqref{eq:linear} yields for any $t>0$  
\begin{eqnarray*}
	\mathbb{P}\left\{ \vert  	\widehat{W}_{\phi} -  W_{\phi} \vert > t  \right\} &\leq & 2 \exp\left\{-\frac{2pNt^2}{\lVert \phi \rVert_{\infty}^2}\right\},\\
	\mathbb{P}\left\{ \frac{1}{n} \left| \widehat{V}_{n}^X -  \EE\left[\widehat{V}_{n}^X\right] \right| > t  \right\} &\leq & 2 \exp\left\{-\frac{2Nt^2}{p\lVert \phi' \rVert_{\infty}^2}\right\},\\
	\mathbb{P}\left\{ \frac{1}{n} \left|  \widehat{V}_{m}^Y -  \EE\left[\widehat{V}_{m}^Y\right] \right| > t  \right\} &\leq & 2\exp\left\{-\frac{2Nt^2}{(1-p)\lVert \phi' \rVert_{\infty}^2}\right\}~.
\end{eqnarray*}

Then, the application of the four tail bounds above with threshold $t/4$, combined with expansion \eqref{eq:linear} and the union bound, concludes the proof.
\end{proof}

	\subsection{Proof of Lemma \ref{lem:cbrem}}
	Consider the nonsymmetric bounded kernels on $\mathbb{R}^2$:
	\begin{eqnarray*}
	 k (x,x') &=& \mathbb{I}\{x'\leq x\} (\phi' \circ \bF)(x)~,\\
		 \ell (x,y)& =&\mathbb{I}\{y\leq x\} ( \phi' \circ \bF)(x)~.
\end{eqnarray*} 
The remainder process is decomposed as follows

\begin{equation*}
	\mathcal{R}_{n,m} = \widehat{R}_{n,m} + \frac{n(n-1)}{N+1}  U_{n} (k)+\frac{nm}{N+1}  U_{n,m} (\ell)+  \widehat{T}_{n,m}~,
\end{equation*}
where
\begin{multline*}
\widehat{R}_{n,m}   = \left(\frac{n}{N+1} -p\right)\sum_{i=1}^n \left(\mathbb{E}[k(X_i, X) \mid X_i] -  \mathbb{E}[k(X, X')]\right) \\
+  \left( \frac{m}{N+1}-1+p\right)\sum_{i=1}^n  \left( \mathbb{E}[k(X_i, Y) \mid X_i] -  \mathbb{E}[k(X, Y)]\right)
\end{multline*}
$(X,X')$ and $Y$ denoting two independent $\rv$'s with distributions $G \otimes G$ and $H$ respectively, independent from the $X_i$'s, $ U_{n}(k) $ is the one-sample degenerate $U$-statistic of order $2$ based on $\{X_1, \ldots, X_n\}$ with kernel $k$, $ U_{n,m}(\ell) $ is the two-sample degenerate $U$-statistic  of degree $(1,1)$ based on the two samples $\{X_1, \ldots,X_n\}$ and  $\{Y_1, \ldots, Y_m\}$ with kernel $\ell$, $\widehat{T}_{n,m}$ is the Taylor-Lagrange integral remainder of the expansion  of $\phi$ at order $2$, namely

\begin{equation}\label{eqlem:remtaylor}
	\widehat{T}_{n,m} =  \sum_{i=1}^n \int_{\bF(X_i)}^{N\widehat{F}_{N}(X_i)/(N+1)} \left(\frac{N\widehat{F}_{N}(X_i)}{N+1}-u\right) \phi''(u)\dd u~.
\end{equation}
\medskip

\noindent We clearly have
\begin{equation*}
\vert	\mathcal{R}_{n,m}\vert \leq \vert \widehat{R}_{n,m}\vert +p^2N \vert U_{n} \vert+ p(1-p)N\vert U_{n,m} \vert + \vert \widehat{T}_{n,m}\vert~.
\end{equation*}

 We now exhibit a tail bound for each of the terms involved in the sum on the right hand side of the bound above.  First,  observing that the mappings $x\mapsto \mathbb{E}[k(x, X) ]  $ $= G(x) \phi' \circ \bF(x)  $ and 
 $x \mapsto \mathbb{E}[k(x, Y) ]  = H(x)  \phi' \circ \bF(x) $ are bounded  by 
 $\lVert \phi' \rVert_{\infty}$,  that $\vert n/(N+1) -p\vert \leq 1/N$ and $\vert m/(N+1) -(1-p)\vert \leq 1/N$,    Hoeffding's inequality applied twice and combined with the union bound directly yields, for any $t>0$,
\begin{equation}\label{pfeq:cbrem0}
	\mathbb{P}\left\{  \frac{1}{n}	\vert	\widehat{R}_{n,m} \vert  > t/4  \right\} \leq 4 \exp\left\{- \frac{pN^3t^2}{32 \lVert \phi' \rVert_{\infty}^2}\right\}.
\end{equation}
	
Tail bounds for the two degenerate $U$-processes are provided by Lemma 27 in \cite{CleLimVay21}, recalled below for clarity.

\begin{lemma} \label{lemma:hoefS} Let $P$ and $Q$ be two probability distributions on measurable spaces $\X$ and $\Y$ respectively. Consider the degenerate two-sample $U$-statistic of degree $(1,1)$ with a bounded kernel $\theta:\X \times \Y \rightarrow \mathbb{R}$ based on the independent $\iid$ random samples $\{X_1,\;  \ldots,\;  X_n \}$ and  $\{Y_1,\;  \ldots,\; Y_m\}$, drawn from $P$ and $Q$ respectively.
	For all $t >0$, we then have:
	\begin{equation*}
		\mathbb{P} \left\{ U_{n,m}(\theta)  \geq t  \right\} \leq   \exp\left\{ -\frac{nmt^2}{32c_{\theta}^2}\right\}~,
	\end{equation*}
	where $c_{\theta}=\sup_{(x,y)\in \X \times \Y}\vert \theta(x,y) \vert<+\infty$.
\end{lemma}
	
A direct application of Lemma \ref{lemma:hoefS} to $U_{n,m}$ yields, for all $t>0$,
	
\begin{equation}\label{pfeq:cbrem1}
	\mathbb{P}\left\{\frac{p(1-p) N}{n}\vert	U_{n,m} \vert  > \frac{t}{4}   \right\} 
	\leq 2 \exp\left\{- \frac{pN^2t^2}{512(1-p) \lVert \phi' \rVert_{\infty}^2}\right\}~,
\end{equation}
and similarly, by virtue of Lemma 3 in \cite{NP87}, we have

\begin{equation}\label{pfeq:cbrem2}
	\mathbb{P}\left\{ \frac{p^2N}{n}\vert	U_{n} \vert  > \frac{t}{4}  \right\} \leq 2 \exp\left\{- \frac{N^2t^2}{512 \lVert \phi' \rVert_{\infty}^2}\right\}~.
\end{equation}

Finally, from Eq. \eqref{eqlem:remtaylor} we derive
\begin{multline*}
	\frac{1}{n}\vert \widehat{T}_{n,m} \vert 
	\leq  \Vert \phi''\Vert_{\infty} \left( \sup_{t \in\mathbb{R} }  \left(\widehat{F}_{N}(t) -\bF(t) \right)^2 
	+  \frac{ 1}{(N+1)^2}\right)\\
	\leq 3  p^2 \Vert \phi''\Vert_{\infty} \sup_{t \in\mathbb{R} }  \left(\widehat{G}_{n}(t) -G(t) \right)^2  
	+ 3 (1-p)^2 \Vert \phi''\Vert_{\infty} \sup_{t \in\mathbb{R} }  \left(\widehat{H}_{m}(t) -H(t) \right)^2   + \frac{13\Vert \phi''\Vert_{\infty}}{N^2}~.
\end{multline*}

Considering the terms involved on the right hand side of the bound above, by applying the classic Dvoretzky–Kiefer–Wolfowitz inequality twice and  noticing that the third term is negligible compared to the first two terms, we obtain:
	\begin{equation}\label{pfeq:cbT}
		\mathbb{P}\left\{ \frac{1}{n}\vert \widehat{T}_{n,m} \vert > \frac{t}{4}  \right\} \leq 4  \exp\left\{- \frac{Nt}{12\kappa_p \lVert \phi'' \rVert_{\infty}}\right\},
	\end{equation}
	where $\kappa_p= \max(p, 1-p) $. Combining the union bound with equations \eqref{pfeq:cbrem0}, \eqref{pfeq:cbrem1}, \eqref{pfeq:cbrem2} and \eqref{pfeq:cbT} concludes the proof.
	
	\subsection{Proof of Theorem \ref{thm:typeI}}
	
	It suffices to observe that we almost-surely have
	\begin{multline*}
	\mathbb{P}_{\mathcal{H}_0}\left\{\frac{1}{n''}\widehat{W}^{\phi}_{n'',m''}(\widehat{s})>\int_0^1\phi(u)\dd u+q_{n'',m''}^{\phi}(\alpha)   \bigm\vert  \bX_1,\; \ldots,\; \bX_{n'},\; \bY_1,\; \ldots,\; \bY_{m'}  \right\}\\
	= \mathbb{P}_{\mathcal{H}_0} \left\{  \frac{1}{n''} \widehat{W}_{n'',m''}>\int_0^1\phi(u)\dd u+q_{n'',m''}^{\phi}(\alpha)  \right\}\leq \alpha
	\end{multline*}
	and take the expectation in the equation above $\wrt$ $\{\bX_1,\; \ldots,\; \bX_{n'},\; \bY_1,\; \ldots,\; \bY_{m'}\}$ .
	
	\subsection{Proof of Proposition \ref{prop:rate_quantile}}\label{app:proofprop8}
	
	Notice that, by virtue of Theorem \ref{thm:tail}, we have
	\begin{equation*}
	\mathbb{P}_{\mathcal{H}_0}\left\{    \frac{1}{n}\widehat{W}^{\phi}_{n,m}-\int_0^1\phi(u)\dd u>\sqrt{\frac{\log(18/\alpha)}{CN}}\right\}\leq \alpha~,
	\end{equation*}
	which yields the desired bound.

	\subsection{Proof of Theorem \ref{thm:typeII}}\label{app:prooftypeII}
	Let $\alpha\in(0,1)$ and $\widehat{s} \in \S_0$ be a solution of \textit{Step 1} based on the two samples  $\mathscr{D}_{n',m'}=\{\bX_1,\; \ldots,\; \bX_{n'}\}	\cup\{\bY_1,\; \ldots,\; \bY_{m'}\}	$. First, observe that, for all $(H,G)\in \mathcal{H}_1(\varepsilon)$, we have:

\begin{multline}\label{eq:dec2}
\mathbb{P}_{H,G}\left\{ \Phi^{\phi}_{\alpha}(\D_{n'',m''}\left(\widehat{s})\right) =0  \right\}
 = \mathbb{P}_{H,G}\left\{\frac{1}{n''}\widehat{W}^{\phi}_{n'',m''}(\widehat{s})-\int_{0}^1\phi(u)\dd u \leq q_{n'',m''}^{\phi}(\alpha)   \right\} \\
\leq \mathbb{P}_{H,G}\left\{  2 \sup_{s\in \S_0}\left\vert \frac{1}{n'} \widehat{W}_{n',m'}(s)-W_{\phi}(s)\right\vert + \left\vert \frac{1}{n''} \widehat{W}^{\phi}_{n'',m''}(\widehat{s})-W_{\phi}(\widehat{s}) \right\vert   \geq \varepsilon - \delta- \sqrt{\frac{\log(18/\alpha)}{CN''}} \right\} \\
\leq \mathbb{P}_{H,G}\left\{  2 \sup_{s\in \S_0}\left\vert\frac{1}{n'}\widehat{W}_{n',m'}(s)-W_{\phi}(s)\right\vert   \geq  \frac{\varepsilon - \delta}{4}  \right\}
+\mathbb{P}_{H,G}\left\{  \left\vert \frac{1}{n''} \widehat{W}^{\phi}_{n'',m''}(\widehat{s})-W_{\phi}(\widehat{s}) \right\vert   \geq \frac{\varepsilon - \delta}{4}  \right\}~,
\end{multline}
where 
$C=8^{-1}\min\left(p/\lVert \phi \rVert_{\infty}^2, (p \lVert \phi' \rVert_{\infty}^2)^{-1}, ((1-p) \lVert \phi' \rVert_{\infty}^2)^{-1} \right) $,
using decomposition \eqref{eq:dec1} above and the bound
\begin{equation*}
W^*_{\phi}-W_{\phi}(\widehat{s}) \leq  2 \sup_{s\in \S_0}\left\vert \frac{1}{n'} \widehat{W}_{n',m'}(s)-W_{\phi}(s)\right\vert +\delta~,
\end{equation*}
as well as Proposition \ref{prop:rate_quantile}, Definition \ref{def:depart}, the union bound and $N''\geq 4\log(18/\alpha)/(C(\varepsilon-\delta)^2)$. By applying Theorem \ref{thm:tail} to the univariate samples:
$$
\{\widehat{s}(\bX_{1+n'}),\; \ldots,\; \widehat{s}(\bX_n)\} \text{ and } 	\{\widehat{s}(\bY_{1+m'}),\; \ldots,\; \widehat{s}(\bY_m)\}~,
$$
when conditioning on  $\mathscr{D}_{n',m'}=\{\bX_1,\; \ldots,\; \bX_{n'}\}	\cup \{\bY_1,\; \ldots,\; \bY_{m'}\}	$ (and consequently the scoring function $\widehat{s}$ produced at \textit{Step 1}), we obtain that we almost-surely have

\begin{multline*}
\mathbb{P}_{H,G}\left\{  \left\vert \frac{1}{n''}   \widehat{W}^{\phi}_{n'',m''}(\widehat{s})-W_{\phi}(\widehat{s}) \right\vert  \geq \frac{1}{2} \left( \varepsilon - \delta- \sqrt{ \frac{\log(18/\alpha)}{CN''}}\right)  \bigm\vert  \mathscr{D}_{n',m'} \right\} \\ 
\leq 18 \exp\left(-\frac{CN''\left( \varepsilon - \delta\right)^2}{16}\right)~.
\end{multline*}

By taking the expectation, the bound above yields
\begin{equation}\label{eq:1part}
\mathbb{P}_{H,G}\left\{  \left\vert \frac{1}{n''} \widehat{W}^{\phi}_{n'',m''}(\widehat{s})-W_{\phi}(\widehat{s}) \right\vert   \geq \frac{1}{2}\left( \varepsilon - \delta\right) \right\}
\leq  18 \exp\left\{-\frac{CN''\left( \varepsilon - \delta\right)^2}{16}\right\}~.
\end{equation}

We also recall the following result providing a tail bound for the maximal deviations between the $W_{\phi}$-ranking performance and its empirical version, where the detailed constants are in the corresponding proof.

\begin{theorem}\label{thm:concentration}(\cite{CleLimVay21}, Theorem 5)
Suppose that the assumptions of Theorem \ref{thm:typeII} are fulfilled. Then, there exist constants $C_1, \; C_2 \geq 24$, depending on $(\phi, \; \V)$, such that for all $C_4\geq C_1$ depending on $\phi$, and for all $t>0$: 

\begin{equation*}
\mathbb{P}\left\{ \sup_{s \in \S_0} \bigg| \frac{1}{n'}  \widehat{W}^{\phi}_{n',m'}(s)- W_{\phi}(s)  \bigg| > t \right\}
\leq C_2\exp\{ -C_3pN't^2\}~,
\end{equation*}
provided that $C_1 /  \sqrt{pN'}  \leq t \leq C_4(p\wedge  (1-p))$, where $C_3 = \log(1+C_{4}/(4C_{1}))/(C_2C_4)>0$ depends on $\phi, \; \V$. 
\end{theorem}

Applying the theorem above with $t=(\varepsilon-\delta)/8=C_4(p\wedge  (1-p))$, we get:
\begin{multline}\label{eq:2part}
\mathbb{P}_{H,G}\left\{  2 \sup_{s\in \S_0}\left\vert\frac{1}{n'} \widehat{W}_{n',m'}(s)-W_{\phi}(s)\right\vert   \geq \frac{\varepsilon - \delta}{4} \right\} \\
\leq C_2
 \exp\left\{ -\frac{N'p(p\wedge  (1-p))}{8C_2}( \varepsilon - \delta)  \left(1+\frac{\varepsilon-\delta}{32C_1(p\wedge  (1-p))}\right) \right\}~,
\end{multline}
as soon as $N'  \geq 16C_1^2/(p(\varepsilon - \delta)^2)$.
Combining \eqref{eq:2part} with \eqref{eq:1part} and \eqref{eq:dec2}, we obtain the desired bound
\begin{multline*}
\mathbb{P}_{H,G}\left\{ \Phi^{\phi}_{\alpha}(\D_{n'',m''}\left(\widehat{s})\right) =0  \right\}\\
\leq 18 \exp\left\{\frac{- CN''\left( \varepsilon - \delta\right)^2}{16}\right\}
+  C_2\exp\left\{ -\frac{N'p(p\wedge  (1-p))}{8C_2}( \varepsilon - \delta)  \log\left(1+\frac{\varepsilon-\delta}{32C_1(p\wedge  (1-p))}\right)  \right\}~.
\end{multline*}

	\section{Additional Numerical Experiments}\label{sec:app_expe}
	\noindent{\bf Exact parameters for the  synthetic datasets.}
\begin{itemize}
		\item[(L1-)] For $d = 4$, the diagonals of $\Sigma$ are  $(2, 6, 1, 5)$, $(-1, 0, 0)$, $(-1, 0)$, $(-1)$; $d=6$. For  $d = 6$, the matrix diagonals are equal to  $(2, 6, 1, 5, 4, 3)$, 
		$(-1, 0, 0, 0, 0)$, $(-1, 0, 0, )$, $(-1, 0, 0)$, $(-1, 0)$, $(-1)$.
		\item[(L1+)] For $d = 4$, the diagonals of $\Sigma$ are  $(6, 4, 5, 3)$, 
		$(-2, 4, 2)$, $(-3, 0)$, $(-2)$; $d=6$. For  $d = 6$, the matrix diagonals are equal to $(6, 5, 5, 3, 2, 3)$, 
		$(-2, 4, 2, 1, 0)$, $(-3, 0, 0, 1)$, $(-2, 1, 1)$, $(-3, 2)$, $(-2)$.
\end{itemize}

\noindent{\bf $\roc^*$ curves for the location and the scale models.}

\begin{figure}[ht!]
		\centering
	\begin{tabular}{cc}
	\parbox{.45\textwidth}{	
	\includegraphics[width=.4\textwidth]{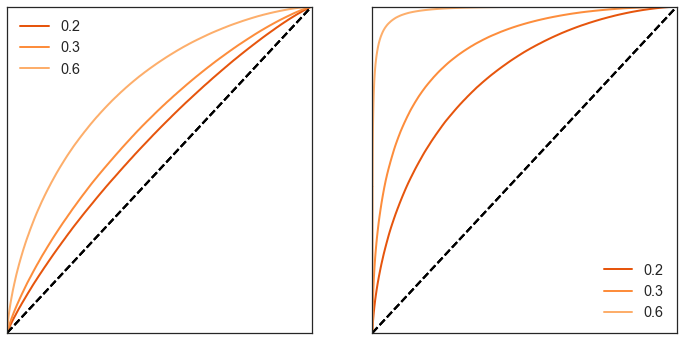}\\
	{a. (L1-), left: $d=4$, right: $d=6$ }\\
	\includegraphics[width=.4\textwidth]{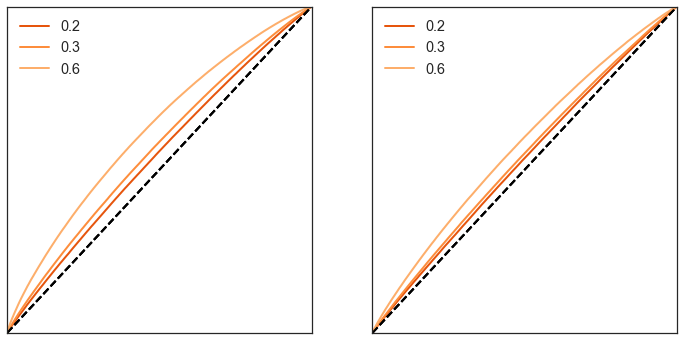}\\
	{b. (L1+), left: $d=4$, right: $d=6$}
}
	\parbox{.45\textwidth}{	
	\includegraphics[width=.4\textwidth]{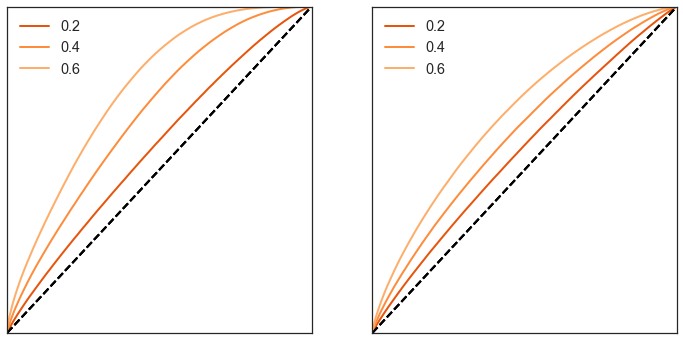}\\
	{c. (S1), left: $d=3$, right: $d=20$}\\
	\includegraphics[width=.4\textwidth]{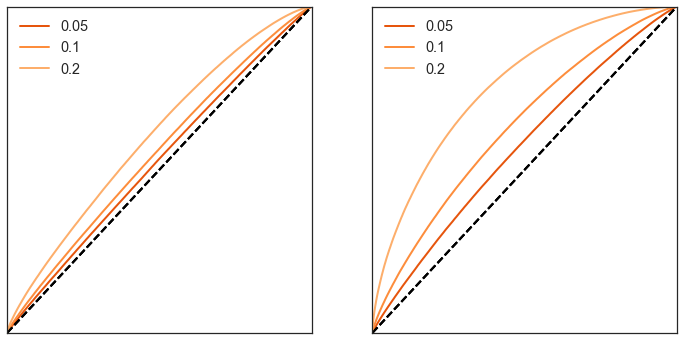}\\
	{d. (S2), left: $d=4$, right: $d=20$}
}
\end{tabular}
	\caption{ True $\roc$ curves ($\roc^*$) for the location and the scale models depending on the discrepancy parameter $\varepsilon \in\{0.2, 0.3, 0.6 \}$ for (L1+) and (L1-), $\varepsilon \in\{0.2, 0.4, 0.6 \}$ for (S1) and $\varepsilon \in\{0.05, 0.1, 0.2 \}$ for (S2).}
	\label{fig:rocstarlocscale}
\end{figure}

\noindent{\bf Additional numerical experiments.}\label{subsec:numexpesup}
The following figures show the numerical results of some experiments exposed in Section \ref{sec:num} but where the score-generating function implemented for the \textit{Step 2} is RTB: $\phi:u\mapsto u\mathbb{I}\{u\geq u_0\}$, with $u_0\in \{0.7, 0.8, 0.9\}$. We used \texttt{rForest}  for \textit{Step 1} as bipartite ranking algorithm.  We show the graphs of the empirical rejection rate of the homogeneity assumption under the alternatives  $\wrt$ the level of the test $\alpha\in (0,1)$.

\begin{figure}[ht!]
	\centering
			\includegraphics[width=.8\textwidth]{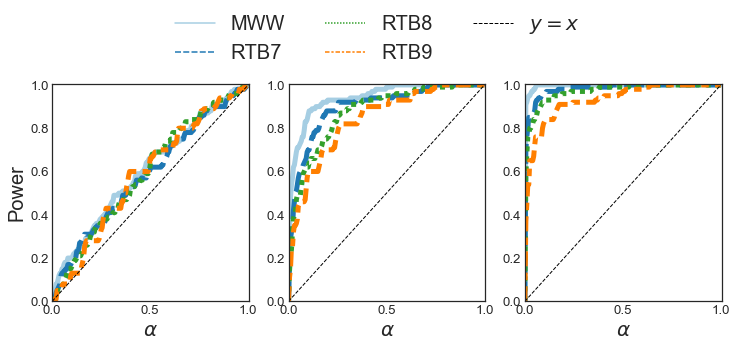}\\
			{\small a. (L1-) , $\varepsilon \in \{0.02, 0.05, 0.08\}$ from left to right}\\
			\includegraphics[width=.8\textwidth]{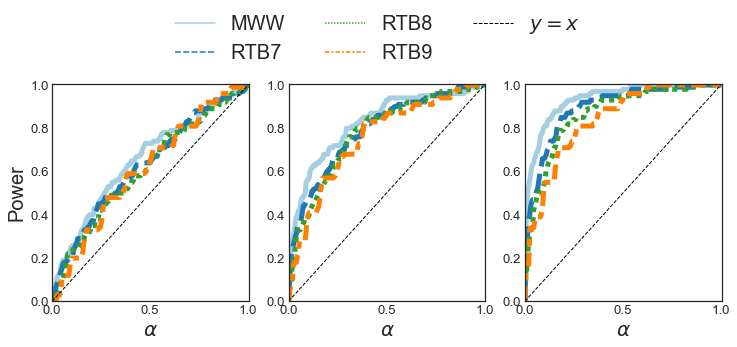}\\
			{\small b. (L1+), $\varepsilon \in \{0.05, 0.08, 0.1\}$ from left to right}
	\caption{  Rejection rates of the homogeneity assumption under alternatives $\wrt$ significance level $\alpha \in (0,1)$ for the models (L1-) (a)  and (L1+) (b) under $\cH_1$. \textit{Step 1} with \texttt{rForest}, \textit{Step 2} with score-generating function RTB   $\phi:u\mapsto u\mathbb{I}\{u\geq u_0\}$ with $u_0\in \{0.7, 0.8, 0.9\}$ $\resp$ corresponding to RTB7, RTB8, RTB9. Comparison of the curves with MWW. }
	\label{fig:RTBL1-L1+treerank}
\end{figure}

\end{document}